\documentclass[11pt,reqno]{amsart}
\usepackage{mathrsfs}
\usepackage{bbm}
\usepackage{amsfonts}
\usepackage{amssymb,amsmath,amscd,amsthm,amsfonts}
\usepackage[dvipdfm,  
            pdfstartview=FitH,
            CJKbookmarks=true,
            bookmarksnumbered=true,
            bookmarksopen=true,
            colorlinks=true, 
            pdfborder=001,   
            linkcolor=green,
            anchorcolor=green,
            citecolor=green
            ]{hyperref}
\usepackage[all]{xy}

\newcommand{\bqa}{\begin{equation}}
\newcommand{\eqa}{\end{equation}}
\newcommand{\bea}{\begin{eqnarray}}
\newcommand{\eea}{\end{eqnarray}}
\newcommand{\bna}{\begin{eqnarray*}}
\newcommand{\ena}{\end{eqnarray*}}
\newcommand{\bma}{\begin{pmatrix}}
\newcommand{\ema}{\end{pmatrix}}

\newcommand{\mk}{\mathfrak}

\def\bz{{\mathbb Z}}
\def\br{{\mathbb R}}

\def\sl2z{SL(2,\bz)}
\def\psl2z{PSL(2,\bz)}

\def\gl2r{GL(2,\br)}

\def\C{\mathbb{C}}

\def\A{\mathbb{A}}
\def\R{\mathbb{R}}
\def\Q{\mathbb{Q}}
\def\Z{\mathbb{Z}}

\def\vol{\mathrm {Vol}}

\newtheorem{lemma}{Lemma}[section]
\newtheorem{thm}[lemma]{Theorem}
\newtheorem{cor}{Corollary}
\newtheorem{prop}[lemma]{Proposition}

\theoremstyle{definition}

\newcommand{\bit}{\begin{itemize}}
\newcommand{\eit}{\end{itemize}}

\begin{document}
\centerline{\Large\bf Central Values of Twisted Base Change
$L$-functions} \centerline{\Large \bf associated to Hilbert Modular
Forms}

\bigskip
\centerline{\large Qinghua Pi}
\address{School of Mathematics and Statistics, Shandong University, Weihai, Weihai 264209, China}
\email{qhpi@sdu.edu.cn}
\numberwithin{equation}{section}

\bigskip

\noindent{\bf Abstract} \, We use the relative trace formula to prove a non-vanishing
result and a subconvexity result for the twisted base change $L$-functions associated
to Hilbert modular forms whose local components at ramified places are some supercuspidal representations.
This generalizes the work of Feigon and Whitehouse in \cite{FW}.
\medskip


\section{Introduction}

Let $F$ be a totally real number field with $\Sigma_\infty$ the set of Archimedean places of $F$.
For $1\leq i\leq 3$,
let $\mk N_i$ be co-prime square-free ideals of the integer ring $\mk o_F$ with $\Sigma_i=\{v<\infty, v\mid\mk N_i\}$.
We assume the  absolute norm $|\mk N_2\mk N_3|$ is not divided by $2$, and
\bea
\#\left(\Sigma_\infty\cup\Sigma_{1}\cup\Sigma_2\right)\equiv 0\bmod 2.\label{condition-1}
\eea
Denote by $\mk N=\mk N_1\mk N_2^2\mk N_3^3$.
Given a tuple of positive integers
${\bf k}=(k_v,v\in\Sigma_\infty)$,
let $\mathcal F^{\mathrm{new}}(2{\bf k},\mk N)$ be the set of automorphic cuspidal
representations of $PGL_2(\A_F)$ which
 are of exact level $\mk N$ and holomorphic
of weight $2{\bf k}$.

Let $E/F$ be a quadratic CM-extension such that $v\in\Sigma_{1}\cup\Sigma_2$
are unramified and inertia in $E$, and $v\in\Sigma_3$ are ramified in $E$.
For $\Omega$ an idele class character of $\A_E^\times$ trivial on $\A_F^\times$,
we consider the central values of
\bna
L(s,\pi_E\otimes\Omega)=L(s,\pi\times\sigma_\Omega),
\ena
where $\pi_E$ is the base change of $\pi$ to an automorphic representation of
$PGL_2(\A_E)$,
and $\sigma_\Omega$ is the induction of $\Omega$ to an automorphic representation of
$GL_2(\A_F)$. For the case $\Sigma_{2}=\Sigma_3=\emptyset$ and $\Omega$ being unramified
at places of $E$ above $\mk N_1$,
by the relative trace formula and period formulas,
Feigon and Whitehouse (\cite{FW}) obtained a formula for
\bna
\sum_{\pi\in\mathcal F^{\mathrm{new}}(2{\bf k},\mk N_1)}
\frac{L(1/2,\pi_E\otimes\Omega)}{L(1,\pi,\mathrm{Ad})},
\ena
and proved a subconvexity bound for $L_{\mathrm{fin}}(1/2,\pi\times\sigma_\Omega)$.
In this paper, we generalize
their results to those $\pi\in\mathcal F^{\mathrm{new}}(2{\bf k},\mk N)$
whose local components $\pi_v$ are some supercuspidal representations of
$PGL_2(F_v)$ at $v\in\Sigma_2\cup\Sigma_3$.
By Bushnell and Henniart (\cite{BuHe}), such supercuspidal representations
are characterized by equivalent classes of
admissible pairs $(E_v/F_v,\theta_v)$.

To state our results, we make the following assumptions.
\bit
\item Assume $\Sigma_2\cup\Sigma_3\neq\emptyset$. For each $v\in \Sigma_{2}\cup\Sigma_3$, let
\bna
\pi_v\leftrightarrow (E_v/F_v,\theta_v)
\ena
be supercuspidal representation of $PGL_2(F_v)$ corresponding to given admissible pair $(E_v/F_v,\theta_v)$.
We refer sections \ref{sec-sigma-2} and \ref{sec-sigma-3}
for details.
\item
Given a tuple of equivalent classes of admissible pairs $\Theta=((E_v/F_v,\theta_v), v\in\Sigma_{2}\cup\Sigma_3)$,
let
\bna
\mathcal F^{\mathrm{new}}(2{\bf k},\mk N,\Theta)
=\{\pi\in\mathcal F^{\mathrm{new}}(2{\bf},\mk N),\quad \pi_v\leftrightarrow (E_v/F_v,\theta_v),v\in\Sigma_2\cup\Sigma_3\}
\ena
be the set of automorphic cuspidal representations
of $PGL_2(\A_F)$ which are of level $\mk N=\mk N_1\mk N_2^2\mk N_3^3$
and weight $2{\bf k}$, and whose local components are characterized by $(E_v/F_v,\theta_v)$
at $v\in\Sigma_{2}\cup\Sigma_3$.
\item Let $\Omega$ be an idele class character of $\A_E^\times$ trivial on $\A_F^\times$.
We assume that
$\Omega_v$ has weight $m_v$ with $|m_v|<k_v$ for $v\in\Sigma_{\infty}$,
$\Omega_v$ is unramified for $v\in\Sigma_{1}$
and
$\Omega_v=\theta_v$
for $v\in\Sigma_{2}\cup\Sigma_3$.
Let $S(\Omega)$ be the set of places of $F$ above which $\Omega$ is ramified, and let
\bna
S'(\Omega)=S(\Omega)-\Sigma_{2}-\Sigma_3.
\ena
\eit

\begin{thm}\label{thm-non-vanihsing}
Let $\Delta_F$ be the discriminant of $F$, $h_F$ the class number of $F$, $d_{E/F}$ the absolute norm of relative discriminant $\mk d_{E/F}$ of $E/F$,
and $c(\Omega)$ the conductor of $\Omega$.
If either
\bna
|\mk N_1\mk N_2^{1+2h_F}\mk N_3^{1+3h_F}|\geq d_{E/F}^{h_F}c(\Omega)^{h_F}
\ena
or
\bna
|\mk N_1\mk N_2^3\mk N_3^4|\geq d_{E/F}c(\Omega)\sqrt{|\Delta_F|},
\ena
we have
\bna
\bma 2{\bf k}-2\\ {\bf k+m}-1\ema\sum_{\pi\in\mathcal F^{\mathrm{new}}(2{\mathbf k},\mk N,\Theta)}\frac{L(1/2,\pi_E\otimes\Omega)}{L(1,\pi,\mathrm{Ad})}
=
\frac{|\mk N_1\mk N_2\mk N_3^2||\Delta_F|^{3/2}}{2^{\#\Sigma_3+[F:\Q]-2}}
L^{S'(\Omega)}(1,\eta).
\ena
Here  $\eta=\eta_{E/F}$ is the quadratic character
associated to $E/F$, and
\bna
\bma 2{\bf k}-2\\ {\bf k+m}-1\ema=\prod_{v\in\Sigma_\infty}
\bma 2 k_v-2\\  k_v+m_v-1\ema.
\ena
\end{thm}
The above result is a generalization of Theorem 1.1 in \cite{FW}. An interesting case is that
\bna
\#(\Sigma_2\cup\Sigma_3)=1,
\ena
where we have  $L(1/2,\pi_E\otimes\Omega)=0$ for those
$\pi\in \mathcal F^{\mathrm{new}}(2{\bf k},\mk N)-\mathcal F^{\mathrm{new}}(2{\bf k},\mk N,\Theta)$.
We remark that for $\Sigma_3=\emptyset$, a much wider result has been obtained
by File, Martin and Pitale in \cite{FMP}, where they determined the local test vectors
for Waldspurger functionals for $GL_2$.

Assume $\Sigma_3=\emptyset$.
Let $\sigma_\Omega$ be automorphic representation of $GL_2(\A_F)$
obtained by automorphic induction of $\Omega$.
We also have the following result,
which generalizes Theorem 1.4 in \cite{FW}.
\begin{thm}\label{thm-subconviexity-0}
Assume $\Sigma_3=\emptyset$.
Let
\bea
c'(\Omega)=\prod_{v<\infty\atop{v\notin\Sigma_2}}c(\Omega_v)=\frac{c(\Omega)}{|\mk N_2|^2}.\label{c-prime-Omega}
\eea
Then for any $\epsilon>0$,
\bna
L_{\mathrm{fin}}(1/2,\pi\times\sigma_{\Omega})\ll_{E,F,{\bf k},\epsilon} |\mk N_1\mk N_2|^{1+\epsilon} c'(\Omega)^{\epsilon}
+|\mk N_1\mk N_2|^{\epsilon} c'(\Omega)^{1/2+\epsilon}
\ena
for all $\pi\in\mathcal F^{\mathrm{new}}(2{\bf k},\mk N,\Theta)$.
\end{thm}
For $\Sigma_3=\emptyset$, the convexity bound of $L_{\mathrm{fin}}(1/2,\pi\times\sigma_\Omega)$ is (see (\ref{formula-convexity-bound}))
\bna
L_{\mathrm{fin}}(1/2,\pi\times\sigma_{\Omega})\ll_{\mathbf{k},\epsilon}(|\mk N_1\mk N_2|d_{E/F}c'(\Omega))^{\frac{1}{2}+\epsilon}.
\ena
We have the following corollary which beats the convexity bound.
\begin{cor}Assume $\Sigma_3=\emptyset$.
For $0\leq t<1/6$ and $\epsilon>0$, we have
\bna
L_{\mathrm{fin}}(1/2,\pi\times\sigma_\Omega)\ll_{E,F,{\bf k},\epsilon}(c'(\Omega)|\mk N_1\mk N_2|)^{\frac{1}{2}-t}
\ena
where $\pi\in\mathcal F^{\mathrm{new}}(2{\bf k},\mk N,\Theta)$ and $\mk N_1\mk N_2$ satisfies
\bna
c'(\Omega)^{\frac{2t+2\epsilon}{1-2t-2\epsilon}}\ll|\mk N_1\mk N_2|
\ll c'(\Omega)^{\frac{1-2t-2\epsilon}{1+2t+2\epsilon}}.
\ena
\end{cor}
Note that the best possible bound appears in the case
 $|\mk N_1\mk N_2|\sim c'(\Omega)^{1/2-\epsilon}$,
where we have
\bna
L_{\mathrm{fin}}(1/2,\pi\times\sigma_\Omega)\ll_{E,F,{\bf k},\epsilon}(c'(\Omega)|\mk N_1\mk N_2|)^{\frac{1}{2}-\frac{1}{6}+\epsilon}.
\ena
We remark that by modifying the global argument in \cite {FW},
one can also obtain other results, such as
explicit formulas in some special cases and equidistribution results on Hecke eigenvalues.
Finally, we refer \cite{KnLi},
\cite{SuTsu1} and \cite{SuTsu2} on other results on applications of the relative trace formula
on Hilbert modular forms.

\medskip

This paper is arranged as follows. In Section \ref{sec-2} we give the notation
and the ideal of the relative trace formula
used by Feigen and Whitehouse in \cite{FW} (or see Jacquet and Chen in \cite{JC}).
In Sections \ref{sec-sigma-2} and \ref{sec-sigma-3}, we give the classification of admissible pairs, construct the associated representations
 for $v\in\Sigma_2$ and $v\in\Sigma_3$ and calculate the local integrals on spectral side.
 The global result on spectral side is obtained in Section \ref{sec-spectral-side}, where
 we use results proved by Martin and Whitehouse in \cite{MW} for $v\notin\Sigma_{2}\cup\Sigma_3$.
The local orbital integrals are computed in Section \ref{sec-local-orbital-integrals}.
In Section \ref{sec-global-orbital-integrals},
we will compute the geometric side of the relative trace formula and prove
Theorems \ref{thm-non-vanihsing} and \ref{thm-subconviexity-0}.

\section{Notation and preliminaries}\label{sec-2}

Let $F$ be a totally real number field
with discriminant $\Delta_F$ and class number $h_F$.
Let $\Sigma_\infty$ be the set of Archimedean places of $F$.
For $1\leq i\leq 3$,
let $\mk N_i$ be co-prime square-free ideals of $\mk o_F$ with $\Sigma_i=\{v<\infty, v\mid\mk N_i\}$.
We assume the  absolute norm $|\mk N_2\mk N_3|$ is not divided by $2$, and
\bna
\#\left(\Sigma_\infty\cup\Sigma_{1}\cup\Sigma_2\right)\equiv 0\bmod 2.
\ena
For $v$ a finite place of $F$,
let $\mk o_{v}$ be the ring of integers in $F_v$,
$U_{v}=\mk o_{v}^\times$ the set of inverse element in $\mk o_{v}$,
$\varpi_{v}$ a uniformizer in $F_v$, and $q_v$
the cardinality of the residue field $k_{F_v}$. For $n\geq 1$, we let
$U_{v}^{n}=1+\varpi_{v}^n\mk o_{v}$.

Let $E/F$ be a quadratic $CM$-extension
with $\eta=\eta_{E/F}$ the associated quadratic character.
We assume that $v\in\Sigma_{1}\cup\Sigma_2$ are unramified inertia in $E$,
and $v\in\Sigma_3$ are ramified in $E$.
Let  $\mk D_{E/F}$ and $\mk d_{E/F}$ be  the different and discriminant of $E/F$, respectively.
We denote $d_{E/F}=|\mk d_{E/F}|$. Moreover, for $v$ a finite place of $F$,
let $E_v=F_v\otimes E$.
Let $\mk o_{E_v}$, $\varpi_{E_v}$, $k_{E_v}$,
$U_{E_v}$ and $U_{E_v}^n$ be in their usual means.

Denote by $\mk N=\mk N_1\mk N_2^2\mk N_3^3$. For ${\bf k}=(k_v,v\in\Sigma_\infty)$
 a triple of positive integers,
let $\mathcal F^{\mathrm{new}}(2{\bf k},\mk N)$ be the set of automorphic cuspidal
representations of $PGL_2(\A_F)$ which are of exact level $\mk N$ and holomorphic
of weight $2{\bf k}$. Moreover, let\bna
\Theta=((E_v/F_v,\theta_v), v\in\Sigma_2\cup\Sigma_3),
\ena
where $(E_v/F_v,\theta_v)$ are admissible pairs
of normalized level $0$ (resp. normalized level $1/2$)
 for $v\in \Sigma_2$ (resp. $v\in\Sigma_3$).
Consider the set \bna \mathcal F^{\mathrm{new}}(2{\bf k},\mk
N,\Theta) =\{\pi\in\mathcal F^{\mathrm{new}}(2{\bf},\mk N),\quad
\pi_v\leftrightarrow (E_v/F_v,\theta_v),v\in\Sigma_2\cup\Sigma_3\},
\ena where $\pi_v\leftrightarrow (E_v/F_v,\theta_v)$ are
supercuspidal representations constructed by the given admissible
pairs for $v\in\Sigma_{2}\cup\Sigma_3$.
We refer Section
 18.2 in \cite{BuHe} for the definition of admissible pair and
 Sections
\ref{sec-sigma-2} and \ref{sec-sigma-3} for more details on
 the classification and construction.

\subsection{Jacquet-Langlands correspondence}
For $E/F$ as above, the condition (\ref{condition-1}) ensures that
 there exists a quaternion division algebra $D$ over $F$ such that the
 set of ramified places of $D$ is
\bna
\mathrm{Ram}(D)=\Sigma_\infty\cup\Sigma_{1}\cup\Sigma_{2}
\ena
and the set of $F$-points is
\bea
D(F)=\left\{\bma \alpha&\epsilon\beta\\\overline\beta&\overline\alpha\ema,\quad\alpha,\beta\in E\right\},\label{def-quaternion-algebra}
\eea
where $\epsilon\in F^\times$ is uniquely determined modulo $N_{E/F}(E^\times)$.
Let $Z$ be the center of $D^\times$ and $G'=D^\times/Z$.

By Jacquet-Langlands correspondence,
for each $\pi\in\mathcal F^{\mathrm{new}}({2\bf k},\mk N,\Theta)$,
there exists an automorphic representation $\pi'=\mathrm{JL}(\pi)$ on $G'(\A_F)$,
whose local components $\pi_v'$ are finite dimensional representations of $G'(F_v)$
at $v\in \mathrm{Ram}(D)$, and $\pi_v'\simeq \pi_v$ at $v\notin \mathrm{Ram}(D)$. More precisely,
\bit
\item for $v\in\Sigma_\infty$, $\pi_v'=\mathrm{JL}(\pi_v)$ is a  $2k_v-1$ dimensional
representation of $G'(F_v)$;
\item for $v\in\Sigma_1$, $\pi_v'=\mathrm{JL}(\pi_v)$ is an one dimensional representation of $G'(F_v)$;
\item for $v\in\Sigma_2$, $\pi_v'=\mathrm{JL}(\pi_v)$ is a two dimensional representation of $G'(F_v)$,
\item for $v\notin \mathrm{Ram}(D)$, $G'(F_v)\simeq PGL_2(F_v)$ and $\pi_v'\simeq \pi_v$.
\eit
Set
\bna
\mathcal F^{'}({2\bf k},\mk N,\Theta)=\{\pi'=\mathrm{JL}(\pi), \quad\pi\in\mathcal F^{\mathrm{new}}({2\bf k},\mk N,\Theta)\}.
\ena
It is a finite subset of $\mathcal A(G')$, the set of automorphic representations on $G'(\A_F)$.

\subsection{The distinguished argument}
For $E$ and $D$ as above,  $E^\times$ is embedded in $D^\times(F)$
as the set of $F$-points
of a torus $T$.
For a Hecke character $\Omega=\prod_v\Omega_v$ on $\A_E^\times$ trivial on $\A_F^\times$,
$\Omega$ can be viewed as a character on $T(\A_F)$.
We choose $\Omega$ such that
\bna
\dim \mathrm{Hom}_{E_v^\times}(\pi_v',\Omega_v)=1
\ena
for $\pi'=\otimes_v\pi_v'\in\mathcal F'(2{\bf k},\mk N,\Theta)$
for all $v$, i.e. $\pi_v'$ is distinguished by $(T_v,\Omega_v)$.
Such $\Omega$ is chosen as follows.
\bit
\item At $v\in \Sigma_\infty$,
$\Omega_v(z)=\left(\frac{z}{\overline z}\right)^{m_v}$
with $m_v$ an integer satisfying $|m_v|<k_v$.
\item At $v\in \Sigma_1$, $\Omega_v$ are unramified.
\item At $v\in \Sigma_{2}\cup\Sigma_3$,
$\Omega_v(z)=\theta_v(z)$,
where $\theta_v$ comes from the given admissible pair $(E_v/F_v,\theta_v)$.
\item At $v\notin \mathrm{Ram}(D)\cup\Sigma_3$, $\Omega_v$ can be either ramified or unramified.
\eit
By the result of Waldspurger and Tunnell (see Proposition 1.1 in \cite{GP}),
the local root number of twisted base change $L$-function is
\bna
\epsilon(\pi_{E_v}\otimes\Omega_v)=\left\{
\begin{aligned}
&\eta_{E_v/F_v}(-1),\quad&& G_v'\simeq PGL_2(F_v),\\
&-\eta_{E_v/F_v}(-1),\quad&&\mbox{otherwise}
\end{aligned}
\right.
\ena
where $\pi_{E_v}$ is the local base change of $\pi_v$.
Thus the global root number is
\bna
\epsilon(\pi_E\otimes\Omega)=(-1)^{\#\mathrm{Ram}(D)}\prod_v\eta_{E_v/F_v}(-1)=1.
\ena

For $v<\infty$,
 let $n(\Omega_v)$ be the least integer such that
$\Omega_v$ is trivial on
\bna
\left(\mk o_{v}+\varpi_v^{n(\Omega_v)}\mk o_{E_v}\right)^\times.
\ena
Let $\mk c(\Omega_v)$ be the norm of the conductor of $\Omega_v$ in $F_v$,
and $c(\Omega_v)$ the absolute norm of $\mk c(\Omega_v)$.
Let $S(\Omega)$ be the set of finite places $v$ of $F$ above which $\Omega$
is ramified. We also denote
\bea
S'(\Omega)=S(\Omega)-\Sigma_{2}-\Sigma_{3}\label{S-prime-Omega}.
\eea

\subsection{Tamagawa measures}\label{subsec-local-tamagawa}

Fix $\psi_0$ a standard additive character on $\A_\Q/\Q$
and let $\psi=\psi_0\circ Tr_{F/\Q}$.
For $v$ a place of $F$, let $dx_v$  be the additive Haar measure on $F_v$
which is self dual with respect to $\psi_v$, and let
\bna
d^\times x_v=L(1,1_{F_v})\frac{dx_v}{|x_v|_v}
\ena
be the multiplicative Haar measure on $F_v^\times$ such that
\bna
\vol(\mk o_{v},dx_v)=\vol( U_{v},d^\times x_v).
\ena
Let $\psi_E=\psi\circ Tr_{E/F}$.
We choose Haar measures $dx_{w}$ on $E_w$ and $d^\times x_{w}$ on $E_w^\times$ in a similar way for $w$ a place of $E$.
This gives
\bna
\prod_{v<\infty}\vol(U_{v},dx_v^\times)=|\Delta_F|^{-1/2},\quad
\prod_{w<\infty}\vol(U_{E_w},dx_w^\times)=|\Delta_E|^{-1/2},
\ena
and for $v\in\Sigma_\infty$,
\bna
\vol(F_v^\times\backslash E_v^\times)=\vol(\R^\times\backslash \C^\times)=2.
\ena
Recall Section 2.1 in \cite{FW}.
For $D$ the quaternion algebra in (\ref{def-quaternion-algebra}) and $v$ a valuation of $F$,
the Tamagawa measure $dg_v$ on $D_v^\times=D^\times(F_v)$ is given by
\bna
dg_v=L(1,1_{F_v})|\epsilon|_v\frac{d\alpha_v d\beta_v}{|\alpha_v\overline\alpha_v-\epsilon\beta_v\overline\beta_v|_v},
\ena
which depends only on the choice of $\epsilon\bmod N(E^\times)$. With respect to the choice
of local Tamagawa measure, we have the following proposition (see section 2.1 in \cite{FW}).
\begin{prop}\label{prop-local-Tamagawa}
For $v<\infty$ and $D_v^\times\simeq GL_2(F_v)$, we have
\bna
\vol(GL_2(\mk o_{v}))=L(2,1_{F_v})^{-1}\vol(U_{v},d^\times x_v)^4.
\ena
\item For $v\in\Sigma_{1}\cup\Sigma_{2}$ and $R_v$ a maximal order of $D_v$,
\bna
\vol(R_v^\times)=L(2,1_{F_v})^{-1}\frac{1}{q_v-1}\vol(U_{v},d^\times x_v)^4.
\ena
\item For $v\in\infty$, we have
\bna
\vol(D_v^\times/Z_v)=4\pi^2 .
\ena

\end{prop}

We choose the global measures as follows.
On $\A_F^\times$, we choose Tammagawa measure $d^\times x=\prod_vd^\times x_v$,
which is the convergent product of the local Tamagawa measures.
Similarly we can define Tamagawa measure on $\A_E^\times$ and we have
\bna
\vol(\A_F^\times E^\times\backslash \A_E^\times)=2L(1,\eta).
\ena
The Tamagawa measure $dg$ on $GL(2,\A_F)$ is the convergent product
multiplying by the factor
\bna
L^{S_0}(2,1_F):=\prod_{v\notin S_0}L(2,1_{F_v}),
\ena
where $S_0$ is a set of places of $F$ out of which everything is unramified
(See page 48 in \cite{JC}).

\subsection{Relative trace formula}
\label{subsec-relative-Jacquet-chen}
We recall trace formula in compact quotient case firstly (see \cite{Ge}).
For $f'\in C_c^\infty(D^\times(\A_F),Z(\A_F))$,
let $R(f')$ act on $\phi\in L^2(G'(F)\backslash G'(\A_F))$ via
\bna
R(f')\phi(x)=\int_{G'(F)\backslash G'(\A_F)} K_{f'}(x,y) \phi(y)dy,
\ena
where $K_{f'}(x,y)$ is the kernel function given by
\bna
K_{f'}(x,y)=\sum_{\gamma\in G'(F)}f'(x\gamma y^{-1}).
\ena
We choose the test function $f'=\prod_vf_v'$ as follows.
\bit
\item
For $v\notin\Sigma_\infty\cup \Sigma_{2}\cup\Sigma_3$,
$\pi_v'$ are unramified representations of $G'_v$.
We choose  $f'_v$  to be Gross-Prasad type test function as in \cite{FW}.
More priecesly, by choosing a maximal order $R_v$ of $D_v$ such that
\bna
E_v\cap R_v=\mk o_{v}+\varpi_v^{n(\Omega_v)}\mk o_{E_v},
\ena
we take $f_v'=1_{Z_vR_v^\times}$.
Assume $\mk o_{E_v}=\mk o_{v}[\tau_v]$. Such maximal order is
(see page 370 in \cite{FW})
\bea
R_v=\left\{\bma\alpha&\epsilon_v\beta\\\overline\beta&\overline\alpha
\ema,
\quad
\begin{aligned}
&\alpha\in\frac{1}{(\overline\tau_v-\tau_v)\varpi_v^{n(\Omega_v)}}(\mk o_v+\varpi_v^{n(\Omega_v)}\mk o_{E_v})\\
&\alpha+\beta\in\mk o_{v}+\varpi_v^{n(\Omega_v)}\mk o_{E_v}
\end{aligned}
\right\}.\label{formula-maximal-order}
\eea

\item
 For $v\in\Sigma_{\infty}\cup\Sigma_2\cup\Sigma_3$, $\pi_v'$ are discrete series representations.
We choose $f_v'$ as
\bna
f_v'(g)=\overline{\langle\pi_v'(g)u_v,u_v\rangle},
\ena
where $u_v$ is a unit vector in
\bea
\pi_v'(\Omega_v)=\{u\in \pi_v',\quad \pi_v'(t)u=\Omega_v(t)u,\forall t\in E_v^\times\}.
\label{formula-distinguished-space}
\eea
\eit
By choosing $f'=\prod_vf_v'$ as above,
we have
\bea
K_{f'}(x,y)=\sum_{\pi'\in\mathcal F'(2{\bf k},\mk N,\Theta)}K_{f',\pi'}(x,y),
\label{tt-spectral-kernel}
\eea
where
$$K_{f',\pi'}(x,y)=\sum_{\phi\in\mathfrak B_{\pi'}}(R(f)\phi)(x)\overline{\phi(y)}$$
with $\mathfrak B_{\pi'}$ an orthonormal basis of $\pi'$.

Denote $[T]=T(F)Z(\A_F)\backslash T(\A_F)=E^\times\A_F^\times\backslash \A_E^\times$.
Consider the integral
\bna
J(f'):=\int_{[T]}\int_{[T]}K_{f'}(t_1,t_2)\Omega(t_1^{-1}t_2)dt_1dt_2.
\ena
By spectral decomposition of $K_{f'}(x,y)$ in (\ref{tt-spectral-kernel}), we have
$$J(f')=\sum_{\pi'\in\mathcal F'(2{\bf k},\mk N,\Theta)} J_{\pi'}(f'),$$
where
\bna
J_{\pi'}(f')=\sum_{\phi\in\mathcal B_{\pi'}}
\int_{[T]}(\pi'(f')\phi)(t)\Omega^{-1}(t)dt
\cdot
\overline{\int_{[T]}\phi(t)\Omega^{-1}(t)dt}.
\ena

In \cite{JC},
Jacquet and Chen gave the canonical decomposition of $J_{\pi'}(f')$ into local integrals, and
proved the non-negativity of $L(1/2,\pi_E\otimes\Omega)$.
We quote their decomposition in the following proposition (see Theorem 2 in \cite{JC}).
\begin{prop}[Jacquet and Chen]\label{prop-Jacquet-Chen}
Let $f'=\prod_vf_{v}'\in C_c^\infty(D^\times(\A_F),Z(\A_F))$
and $\pi'\in\mathcal A(G'(\A_F))$.
Assume $S_0$ is the set of places of $F$ such that out
of which everything is unramified, and $S$ is the set of places in $E$ above $S_0$.
We have
\bna
 J_{\pi'}(f')=\prod_{v\in S_0}J_{\pi'_{v}}(f_{v}')
\times\frac{1}{2}\left(\prod_{v\in S_0\atop{\mathrm{inertia}}}\epsilon(1,\eta_v,\psi_v)2L(0,\eta_v)\right)
\times\frac{L_{S_0}(1,\eta) L^S(1/2,\pi_E\otimes\Omega)}{L^{S_0}(1,\pi,\mathrm{Ad})},
\ena
where
\bna
J_{\pi'_v}(f'_{v})
=\sum_{W\in\mathscr W_{\pi'}}\int_{T(F_{v})/Z_v} \pi'_v(f_{v}')W(a) \Omega_{v}^{-1}(a)da
\overline{\int_{T(F_{v})/Z_v}W(a)\Omega_{v}^{-1}(a)da}
\ena
if $v$ is split in $E$,
and
\bna
J_{\pi'_v}(f'_{v})
=
\int_{G'_{v}} f'_{v}(g)\langle\pi'_v(g)u,u\rangle dg
\ena
if $v$ is inertia in $E$. Here $\mathscr W_{\pi'}$  is the Whittaker model of $\pi'$
 and
 $u$ is the unit vector in $\pi_v'(\Omega_v)$.
\end{prop}
Based on Jacquet and Chen's factorization, the spectral side of $J(f')$
follows from the explicit calculation of local integrals $J_{\pi'_v}(f_v')$ for each $v$.
We will calculate $J_{\pi_v'}(f_v')$ for $v\in\Sigma_2\cup\Sigma_3$ in Sections \ref{sec-sigma-2} and \ref{sec-sigma-3}.
These results together with the work of Martin and Whitehouse in \cite{MW} give
the formula on  spectral side of $J(f')$ in Proposition \ref{prop-spectral-side}
in section \ref{sec-spectral-side}.

Consider the geometric side of $J(f')$.
By Feigon and Whitehouse in \cite{FW} (or see Jacquet and Chen in \cite{JC}),
a set of representatives for $E^\times\backslash G'(F)/E^\times$ is
\bna
\left\{\bma 1&\\&1\ema,\bma&\epsilon\\1\ema\right\}\bigcup \left\{\bma 1&\epsilon x\\\overline x&1\ema,\quad x\in E^\times/E^1\right\}.
\ena
It follows that\bna
J(f')=\vol(\A_F^\times E^\times\backslash \A_E^\times) \left[I(0,f')+\delta(\Omega^2)I(\infty,f')\right]
+\sum_{\xi\in\epsilon N(E^\times)}I(\xi,f'),
\ena
where $I(0,f')$ and $I(\infty,f')$ are singular orbital integrals given by
\bea
I(0,f')&=&\int_{\A_F^\times\backslash\A_E^\times}f'(t)\Omega(t)dt,\label{I-0-f-prime}\\
\nonumber I(\infty,f')&=&\int_{\A_F^\times\backslash\A_E^\times}f'\left(t\bma &\epsilon\\1\ema\right)\Omega(t)dt,
\eea
and for $1\neq \xi=\epsilon x\overline x$, $I(\xi,f')$ is
the regular orbital integral given by
\bea
I(\xi,f')=\int_{\A_F^\times\backslash \A_E^\times}
\int_{\A_F^\times\backslash \A_E^\times}f'\left(t_1\bma 1&\epsilon x\\\overline x&1\ema t_2\right)\Omega(t_1t_2)dt_1dt_2.\label{I-xi-f-prime}
\eea
Note that we assume $\Sigma_{2}\cup\Sigma_{3}\neq\emptyset$. It implies $\delta(\Omega^2)=0$ and thus
\bea
J(f')=\vol(\A_F^\times E^\times\backslash \A_E^\times)I(0,f')
+\sum_{\xi\in\epsilon N(E^\times)}I(\xi,f')\label{formula-geometric-side-pre}.
\eea

The orbital integrals canonically decompose into product of local orbital integrals.
We will calculate the local orbital integrals at places $v\in\Sigma_{2}\cup\Sigma_3$ in Section \ref{sec-local-orbital-integrals}.
The global orbital integrals are discussed in Section \ref{sec-global-orbital-integrals}.

\section{Local integrals on spectral side at $v\in \Sigma_2$}\label{sec-sigma-2}
For $v\in\Sigma_2$,
$E_v/F_v$ is a non-split unramified quadratic extension,
$D_v$ is a quaternion division algebra and
$\pi_v'=\mathrm{JL}(\pi_v)$,
where $\pi_v$ is a depth-zero supercuspidal representation of $PGL_2(F_v)$.
For a classification of depth-zero supercuspidal representations of $GL_2(F_v)$, we
refer to \cite{KnRa}.
By Bushnell and Henniart (\cite{BuHe}), $\pi_v$ can also be characterized by admissible pairs $(E_v/F_v,\theta_v)$
of normalized level zero.

\subsection{Admissible pairs of normalized level zero}\label{subsec-admissible-pair-normalized-level-0}
We recall Section 18 in \cite{BuHe} on admissible pairs which correspond to
 supercuspidal representations of $PGL_2(F_v)$ of level $\mk p_v^2$.
Let $(E_v/F_v,\theta_v)$ be an admissible pair of normalized level 0.
Then
$E_v/F_v$ is unrmaified quadratic extension,
$\theta_v$ is a character of $E_v^\times$ trivial on $U_{E_v}^1$,
and $\theta|_{U_{E_v}}$ does not factor
through the norm map. Moreover we assume $\theta_v$ is trivial on $F_v^\times$
so that
$\theta_v$ is determined by its restriction on $k_{E_v}^\times$.

Fix an isomorphism
$k_{E_v}^\times\simeq\mu(q^2-1)$
with $\mu(q^2-1)$  the group of $(q^2-1)$-th roots of unity.
The restriction of $\theta_v$ on  $\mu(q^2-1)$ is of the form
\bna
\theta_{v,a}:e\left(\frac{1}{q^2-1}\right)\mapsto e\left(\frac{a}{q^2-1}\right)
\ena
for some $a\bmod q^2-1$, with $(q-1)\mid a$ and $(q+1)\nmid a$.
Here we use the notation $e\left(x\right)=\exp(2\pi i x)$.

Moreover,
 $(E_v/F_v,\theta_{v,a})\simeq(E_v/F_v,\theta_{v,aq})$ in the sense
 that they correspond to the equivalent supercuspidal representations
 of $PGL_2(F_v)$. There is $(q-1)/2$ number of such equivalent
 classes of admissible pairs.

\subsection{Representations associated to admissible pairs}
Fixing $(E_v/F_v,\theta_{v,a})$ as above, let $\pi_v\leftrightarrow (E_v/F_v,\theta_{v,a})$ be
the depth-zero supercuspidal representation of $PGL_2(F_v)$. It is of level $\mk p_v^2$. We can construct
the representation $\pi_v'=\mathrm{JL}(\pi_v)$ directly as follows (see Section 54.2 in \cite{BuHe}).

Let $\varpi_{D_v}$ be a prime element in $D_v$ with $(v\circ \mathrm{Nrd})(\varpi_{D_v})=1$,
where $\mathrm{Nrd}$ is the reduced norm on $D_v$.
Choose a maximal order $R_v$ in $D_v$ such that
\bna
R_v\cap E_v=\mk o_{E_v},
\ena
and denote by $U_{D_v}=R_v^\times$. One has
$D_v^\times=\langle\varpi_{D_v}\rangle\ltimes U_{D_v}$. Moreover, for $n\geq 1$,
\bna
U_{D_v}^n=1+\varpi_{D_v}^{n}R_v,
\ena
which are compact subgroups of $U_{D_v}$.

Extend $\theta_{v,a}$ of $E_v^\times$ to be a character of $E_v^\times U_{D_v}^1$ and let
\bna
\pi_v':=\mathrm{Ind}_{E_v^\times U_{D_v}^1}^{D_v^\times}\theta_{v,a}.
\ena
Note that $[D_v^\times:E_v^\times U_{D_v}^1]=2$. We know $\pi_v'$ is
 a two dimensional irreducible representation of $G_v'$ such that $\pi_v'=\mathrm{JL}(\pi_v)$,
where $\pi_v\leftrightarrow (E_v/F_v,\theta_{v,a})$. Moreover,
\bna
\pi'_v|_{E_v^\times}=\theta_{v,a}\oplus\theta_{v,aq}.
\ena
Thus $\pi_{v}'$ is canonically distinguished by $(E_v^\times,\Omega_v)$ for $\Omega_v=\theta_{v,a}$,  and
\bna
\epsilon(\pi_{E_v}\otimes\Omega_v)=-\eta_{E_v/F_v}(-1).
\ena
\subsection{The local integrals at $v\in\Sigma_2$}
The test function $f_v'$ is
\bea
f_v'(g)=\overline{\langle \pi_v'(g)u,u\rangle}
=\left\{
\begin{aligned}
&\theta_{v,a}^{-1}(z),\quad &&g=zu,z\in E_v^\times, u\in U_{D_v}^1,\\
&0,\quad&&\mbox{otherwise}
\end{aligned}
\right.\label{testfunc-v-in-Sigma-2}
\eea
where $u$ is a unit vector in $\pi'_v(\theta_v)$ in (\ref{formula-distinguished-space}). Thus the local integral is
\bna
J_{\pi'_v}(f_v')=\int_{G_v'}f_v'(g)\langle\pi_v(g)u,u\rangle dg
=\vol(E_v^\times U_{D_v}^1/F_v^\times).
\ena

We need to express the local integral via critical $L$-values.
It is well known that
\bna
L(s,\pi_v,\mathrm{Ad})=L(s,\eta_{E_v/F_v})=(1+q_v^{-s})^{-1}.
\ena
For $\pi_v'=\mathrm{JL}(\pi_v)$ with  $\pi_v\leftrightarrow (E_v/F_v,\theta_v)$,
by local Langlands correspondence,
there exists a two-dimensional representation $\rho_v$ of Weil group $\mathcal W_F$
corresponding to $\pi_v$ of $PGL_2(F_v)$.
Such $\rho_v$ is
\bna
\rho_v=\mathrm{Ind}_{\mathcal W_E}^{\mathcal W_F}(\theta_{v,a}\chi_v),
\ena
where $\chi_v$ is an unramified quadratic character of $E^\times_v$ with $\chi_v(\varpi_v)=-1$ (see Definition 34.4 in \cite{BuHe}).
The twisted base change $\pi_{E_v}\otimes \Omega_v$ corresponds to
\bna
\rho_v|_{\mathcal W_E}\otimes\Omega_v=\chi_v\theta_{v,a}^{2}\oplus\chi_v,
\ena
and thus
\bna
L(s,\pi_{E_v}\otimes\Omega_v)=L(s,\rho_v|_{\mathcal W_E}\otimes\Omega_v)=(1+q_v^{2s})^{-1}.
\ena
Therefore, we have the following proposition.
\begin{prop}\label{prop-sigma_2}For $v\in\Sigma_{2}$, by choosing $f_v'$ as in (\ref{testfunc-v-in-Sigma-2}),
one has
\bna
J_{\pi'_v}(f_v')
=\vol(E_v^\times U_{D_v}^1/F_v^\times)\frac{L(1/2,\pi_{E_v}\otimes\Omega_v)}{L(1,\pi_v,\mathrm{Ad})}.
\ena
\end{prop}

\section{Local integrals on spectral side at $v\in\Sigma_3$}\label{sec-sigma-3}
For $v\in\Sigma_3$, $G'_v\simeq PGL_2(F_v)$ and
$\pi_v'\simeq\pi_v\leftrightarrow (E_v/F_v,\theta_v)$ is of level $\mk p_v^3$.
In this case, $\pi_v$ are known as simple supercuspidal representations (see \cite{KnLi2}).
Note that the characteristic
of the residue field of $F_v$ is odd, since $2\nmid |\mk N_2\mk N_3|$.
\subsection{Admissible pairs of normalized level $1/2$}\label{subsec-admissible-pair-normalized-level-1-2}
We recall Section 18 in \cite{BuHe} on admissible pairs which correspond to
 supercuspidal representations of $PGL_2(F_v)$ of level $\mk p_v^3$.
Let $(E_v/F_v,\theta_v)$ be an admissible pair of normalized level $1/2$.
Then  $E_v/F_v$ is a ramified quadratic extension,
$\theta_v$ is a character of $E_v^\times$ which is trivial on $U_{E_v}^2$
and nontrivial on $U_{E_v}^1$, and the restriction of $\theta_v$ to
 $U_{E_v}^1$ does not factor through the norm map.
 Moreover we assume that $\theta_v|_{F_v^\times}=1$.
Such admissible pairs are classified in the following lemma,
which is a result of a series of lemmas in Section 18 in \cite{BuHe}.
\begin{lemma}\label{prop-admissible-pair-sigma-3}
Fix an isomorphism of $k_{F_v}\simeq \Z/q\Z$ and let
\bna
\tilde\psi_v:k_{F_v}\simeq \Z/q\Z\rightarrow S^1,\quad 1\mapsto e(1/q)
\ena
be a non-trivial additive character on $k_{F_v}$. Extend $\tilde\psi_v$ to be a character of $F_v$
and let $\tilde\psi_{E_v}=\tilde\psi_v\circ Tr_{E_v/F_v}$.

Let $\theta_v$ be a character of $E_v^\times$ which is trivial on $F_v^\times U_{E_v}^2$
and non-trivial on $U_{E_v}^{1}$.
Then $\theta_v$ is determined by its value at $\varpi_{E_v}$ and on $U_{E_v}^1/U_{E_v}^2$,
and there exists $\alpha=\varpi_{E_v}^{-1}\beta\in\varpi_{E_v}^{-1}\mk o_{E_v}^\times$ such that
\bna
\theta_v(x)=\tilde\psi_{E_v}(\alpha(x-1)),\quad x\in E_v^\times.
\ena

By realizing $U_{E_v}^1/U_{E_v}^2=\{1+\varpi_{E_v} y,\quad y\in k_{F_v}=\Z/q \Z\}$,
we can denote by $\theta_v=\theta_{v,\pm,2\beta}$ for some $\beta\in \Z/q\Z-\{0\}$.
It is determined by the values
\bna
\theta_{v,\pm,2\beta}(\varpi_{E_v})&=&\pm 1\\
\theta_{v,\pm,2\beta}(1+y\varpi_{E_v})&=&\tilde\psi_v(2\beta y).
\ena
Moreover,
$
\{(E_v/F_v,\theta_{v,+,2\beta}),(E_v/F_v,\theta_{v,+,-2\beta})\}$
and
$\{(E_v/F_v,\theta_{v,-,2\beta}),(E_v/F_v,\theta_{v,-,-2\beta})\}$
are equivalent classes of admissible pairs,
and there are $q-1$ number of inequivalent classes of admissible pairs of
{\it normalized level} $1/2$.

\end{lemma}
\begin{proof}We can assume that
$\mk o_{E_v}=\mk o_{v}[\varpi_{E_v}]$, where $\varpi_{E_v}$
is a root of
$x^2-\varpi_v=0$.
One has $k_{F_v}^\times\simeq \mu(q-1)$ and
\bna
\begin{aligned}
F_v^\times&=\langle\varpi_v\rangle\cdot \mu(q-1)\cdot U_v^{1},\\
E_v^\times&=\langle\varpi_{E_v}\rangle\cdot \mu(q-1)\cdot U_{E_v}^{1}/U_{E_v}^2\cdot U_{E_v}^2.
\end{aligned}
\ena
Note $U_v^1=1+\varpi_v\mk o_v\subset U_E^2$. Thus $\theta_v$ factors through
\bna
\theta_v: E_v^\times\rightarrow E_v^\times/{F_v}^\times U_{E_v}^2=\frac{\langle\varpi_{E_v}\rangle}{\langle\varpi_v\rangle}
\times \frac{U_{E_v}^1}{U_{E_v}^2}\rightarrow S^1,
\ena
where
\bna
U_{E_v}^1/U_{E_v}^2=\{1+\varpi_{E_v} y,\quad y\in k_{F_v}\}.
\ena

By Bushnell and Henniart in \cite{BuHe},
there exists $\alpha=\varpi_{E_v}^{-1}\beta\in\varpi_{E_v}^{-1}\mk o_{E_v}^\times$ such that
\bna
\theta_v(x)=\tilde \psi_{E_v}(\alpha(x-1)),\quad\forall x\in U_{E_v}^1.
\ena
In fact, for $x\in U_{E_v}^1$, by writing
$x= 1+\varpi_{E_v} y+\varpi_{E_v}^2 y'$ with $y\in k_{F_v}$ and $y'\in\mk o_{E_v}$,
one has
\bna
\begin{aligned}
\theta_v(x)&=\tilde\psi_{E_v}(\alpha(x-1))\\
&=(\tilde\psi_v\circ Tr_{E_v/F_v})(\varpi_{E_v}^{-1}\beta(\varpi_{E_v} y+\varpi_{E_v}^2y')))\\
&=\tilde\psi_v(2\beta y),
\end{aligned}
\ena
where $0\neq\beta\in \Z/q\Z$. Thus $\theta_v$ is of the form $\theta_v=\theta_{v,\pm,2\beta}$
for some $\beta\in Z/q\Z-\{0\}$.
Moreover, the restriction of $\theta_{v,\pm,2\beta}$ to $U_{E_v}^1$ does not factor through $N_{E_v/F_v}$,
since
\bna
N_{E_v/F_v}(U_{E_v}^{2})=N_{E_v/F_v}(U_{E_v}^{1}).
\ena

Finally,
$\{(E_v/F_v,\theta_{v,+,2\beta}),(E_v/F_v,\theta_{v,+,-2\beta})\}$
and
$\{(E_v/F_v,\theta_{v,-,2\beta}),(E_v/F_v,\theta_{v,-,-2\beta})\}$
are equivalent classes of admissible pairs (see Section 18 in \cite{BuHe}).
\end{proof}

Consider
\bna
D_v=\left\{\bma\alpha &\beta\\\overline\beta&\overline\alpha\ema,\quad \alpha,\beta\in E_v\right\}.
\ena
It is an $F_v$-algebra generated by
\bna
\left\{\bma\varpi_{E_v}&\\&-\varpi_{E_v}\ema,\quad
\bma&1\\1&\ema\right\}.
\ena
We have an $F_v$-algebra isomorphism $D_v\simeq M_2(F_v)$ given by
\bea
\bma\varpi_{E_v}\\&-\varpi_{E_v}
\ema
\mapsto\bma&1\\\varpi_{v}\ema,
\quad
\bma&1\\1&
\ema
\mapsto\bma1&\\&-1\ema.\label{iso-Dv-M2Fv}
\eea
The image of $\mk o_{E_v}$ is contained in a maximal order $\mk J$ of $M_2(F_v)$, where
\bna
\mk J:=\left[\begin{matrix}\mk o_v&\mk o_v\\\mk p_v&\mk o_v\end{matrix}\right],
\ena
and the image of $\varpi_{E_v}$ in $\mk J$ is
\bna
\varpi_\mk J:=\bma&1\\\varpi_v\ema,
\ena
which is a prime element in $\mk J$ such that $v_\mk J(\varpi_\mk J)=v_{E_v}(\varpi_{E_v})=1$,
where $v_\mk J=v\circ \det$ is the discrete valuation on $\mk J$.
We know
$U_\mk J=\left\{ X\in\mk J, v_\mk J(X)=0\right\}$
is the Iwahori subgroup of $GL_2(F_v)$. Let
\bna
U_\mk J^n=I_2+\varpi_\mk J^n \mk J,\quad n\geq 1.
\ena
The following lemma is a well-known result.
\begin{lemma}\label{prop-U-J-U-J-1} We have
\bna
U_{\mk J}/U_{\mk J}^1\simeq\left\{\bma a&\\&d\ema,\quad a,d\in k_{F_v}^\times\right\}.
\ena
Moreover, a set of representatives of $U_{\mk J}^1/U_{\mk J}^2$ is
\bea
\left\{X=I_2+\bma 0&a\\\varpi_v b&0\ema,\quad a,b\in k_{F_v}\right\}.\label{U-J-1-2}
\eea
\end{lemma}
\subsection{Representations associated to admissible pairs}
We can construct supercuspidal representations $\pi_v'\simeq\pi_v$ of $G_v'$
of level $\mk p_v^3$
via
admissible pairs as in Section 19.3 in \cite{BuHe}.

Given $(E_v/F_v,\theta_{v,\pm,2\beta})$ as in Lemma \ref{prop-admissible-pair-sigma-3},
we have $\alpha=\varpi_{E_v}^{-1}\beta$ such that for $x=1+y\varpi_{E_v}\in U_{E_v}^1/U_{E_v}^2$,
\bna
\theta_{v,\pm,2\beta}(x)=\tilde\psi_{E_v}(\alpha(x-1))=\left(\tilde\psi_{v}\circ Tr_{E_v/F_v}\right)(\beta y).
\ena
Note that the image of $\alpha=\varpi_{E_v}^{-1}\beta$ in $M_2(F_v)$ is
\bna
\alpha \mapsto\varpi_v^{-1}\bma&\beta\\\varpi_v\beta\ema.
\ena
We can extend $\theta_{v,\pm,2\beta}$ to be a character $\lambda_\beta$ of $U_{\mk J}^1$ by
\bna
\lambda_\beta(X)=(\tilde\psi_v\circ Tr)(\alpha (X-I_2)).
\ena
In fact, for $X=I_2 +\bma 0&a\\\varpi_v b&0\ema\in U_\mk J^1/U_\mk J^2$ as in (\ref{U-J-1-2}),
\bna
\lambda_\beta(X)&=&(\tilde\psi_v\circ Tr)\left(\varpi_v^{-1}\bma&\beta\\\varpi_v\beta\ema\bma&a\\\varpi_v b\ema
\right)\\
&=&\tilde\psi_v(\beta(a+b))\\
&=&\theta_{v,\pm,2\beta}\left(1+\varpi_v\frac{a+b}{2}\right).
\ena
So $\lambda_\beta|_{U_{E_v}^1}=\theta_{v,\pm,2\beta}$.

Let $\Lambda$ be a character of of $E_v^\times U_J^1$ such that
\bna
\Lambda|_{E_v^\times}=\theta_{v,\pm,2\beta},\quad\Lambda|_{U_\mk J^1}=\lambda_\beta.
\ena
Note that
$E_v^\times U_\mk J^1=\langle\varpi_{E_v}\rangle\cdot \mu(q-1)\cdot
U_{\mk J}^1/U_{\mk J}^2 \cdot U_\mk J^2$.
$\Lambda$ is determined by its values
\bna
\begin{aligned}
\Lambda(\varpi_{E_v})&=\theta_{v,\pm,2\beta}(\varpi_E)=\pm 1\\
\Lambda(X)&=\theta_{v,\pm,2\beta}\left(1+\varpi_v\frac{a+b}{2}\right),
\end{aligned}
\ena
where $X=I_2+\bma &a\\b\varpi_v&\ema\in U_{\mk J}^1/U_{\mk J}^2$.

By Section 19.3 in \cite{BuHe},
\bea
\pi_v'=\mathrm{c-Ind}_{E^\times U_\mk J^1}^{GL_2(F_v)}\Lambda\label{sigma-3-pi_v}
\eea
is an irreducible supercuspidal representation of $GL_2(F_v)$ with trivial center character.

\subsection{The local integrals}
For $\pi_v'$ in (\ref{sigma-3-pi_v}) and $\Omega_v=\theta_{v,\pm,2\beta}$,
obviously one has
\bna
\dim \mathrm{Hom}_{E_v^\times}(\pi_v',\Omega_v)=1.
\ena
Let $0\neq u$ be a unit vector in $\pi_v'(\Omega_v)$ in (\ref{formula-distinguished-space}).
By Proposition 1.1 in \cite{KnRa},
the test function $f_v'(g)$ is
\bea
f_v'(g)=\overline{\langle\pi_v'(g)u,u\rangle}
=\left\{
\begin{aligned}
&\Lambda^{-1}(g),\quad &&g\in E_v^\times U_{\mk J}^1,\\
&0,\quad&&\mbox{otherwise}.
\end{aligned}
\right.\label{testfunc-v-in-Sigma-3}
\eea
More precisely,  we can write $g\in E_v^\times U_{\mk J}^1$ as
\bna
g=\varpi_{E_v}^m \kappa\cdot \bma 1&a\\\varpi_v b&1\ema\cdot u
\ena
with $\kappa\in \mu(q-1)$,
$\bma 1&a\\\varpi_v b&1\ema\in U_{\mk J}^1/U_{\mk J}^2$ and $u\in U_{\mk J}^2$.
In this case,
\bna
\Lambda(g)=(\pm 1)^m \theta_{v,\pm, 2\beta}\left(1+\varpi_{E_v}\frac{a+b}{2}\right).
\ena

The local distribution is
\bna
J_{\pi'_v}(f_v')=\int_{G_v'/Z_v}f_v'(g)\langle \pi_v'(g)u,u\rangle dg
=\vol(E_v^\times U_\mk J^1/F_v^\times),
\ena
which is the inverse of the formal degree of $\pi_v'$ (see Proposition 1.2 in \cite{KnRa}).
We need to express $J_{\pi'_v}(f_v')$  via critical $L$-values.
By Proposition 34.3 and  Section 34.4 in \cite{BuHe},
for $\pi_v'\simeq \pi_v\leftrightarrow (E_v/F_v,\theta_{v,\pm,2\beta})$,
there exists another admissible pair $(E_v/F_v,\theta'_v)$ with
\bna
\theta'_v|_{U_{E_v}^1}=\theta_{v,\pm,2\beta},\qquad
\theta'_{v}|_{F_v^\times}=\eta_{E_v/F_v},
\quad\theta'_{v}(\varpi_{E_v})= \pm\eta_{E_v/F_v}(\beta)\frac{\tau(\eta_{E_v/F_v},\psi_v)}{q^{1/2}},
\ena
and a character $\chi_v$ of $E_v^\times$ with
\bna
\chi_v|_{U_{E_v}^1}=1,\quad
\chi_v|_{F_v^\times}=\eta_{E_v/F_v},
\quad\chi_v(\varpi_{E_v})=\eta_{E_v/F_v}(\beta)\frac{\tau(\eta_{E_v/F_v},\psi_v)}{q^{1/2}},
\ena
such that $(E_v/F_v,\theta_v'\chi_v)\simeq (E_v/F_v,\theta_{v,\pm,2\beta})$.
Here $\tau(\eta,\psi_v)$ is Gauss sum (see Section 23.6 in \cite{BuHe}).
By the Local Langlands correspondence in Theorem 34.4 in \cite{BuHe},
$\pi_v'$ corresponds to a two-dimensional representation $\rho_v$ of $\mathcal W_F$ given by
\bna
\rho_v:=\mathrm{Ind}_{\mathcal W_E}^{\mathcal W_F}\theta_v'.
\ena
Thus the twisted base change $\pi_{E_v}\otimes \Omega_v$ corresponds to
\bna
\rho_v|_{\mathcal W_E}\otimes\Omega_v=(\theta_{v}')\theta_v\bigoplus(\theta_{v}')^\sigma\theta_v,
\ena
where $\sigma$ is the non-trivial element in $\mathrm{Gal}(E_v/F_v)$. It gives that
\bna
L(s,\pi_{E_v}\otimes\Omega_v)=L(s,\rho_v|_{\mathcal W_E}\otimes\Omega_v)=1.
\ena
Thus we have the following proposition.
\begin{prop}\label{prop-sigma_3}
For $v\in\Sigma_{3}$, by choosing $f_v'$ as in (\ref{testfunc-v-in-Sigma-3}),
one has
\bna
J_{\pi'_v}(f_v')
=\vol(E^\times U_\mk J^1/F^\times)\frac{L(1/2,\pi_{E_v}\otimes\Omega_v)}{L(1,\pi_v,\mathrm{Ad})}.
\ena
\end{prop}

\section{Global integrals on spectral side}\label{sec-spectral-side}
In this section, we recall results on
local integrals $J_{\pi_v'}(f_v')$  for $v\notin\Sigma_2\cup\Sigma_3$ in \cite{MW},
and prove the following proposition.
\begin{prop}\label{prop-spectral-side}
We have
\bna
J(f')=
\frac{4^{[F:\Q]}L_{S'(\Omega)}(1,\eta)^2}{2|\Delta_F|^{2}\sqrt{c(\Omega)d_{E/F}}}
\frac{2^{\#\Sigma_3}}{|\mk N_1\mk N_3|}
\bma 2{\bf k}-2\\ {\bf k+m}-1\ema
\sum_{\pi\in\mathcal F^{\mathrm{new}}(2{\bf k},\mk N,\Theta)}\frac{L(1/2,\pi_E\otimes\Omega)}{L(1,\pi,\mathrm{Ad})},
\ena
where $S'(\Omega)=S(\Omega)-\Sigma_{2}-\Sigma_3$ is in (\ref{S-prime-Omega}), and
\bna
\bma 2{\bf k}-2\\ {\bf k+m}-1\ema=\prod_{v\in\Sigma_\infty}
\bma 2 k_v-2\\ k_v+m_v-1\ema.
\ena
\end{prop}

\subsection{Local integrals at $v\notin\Sigma_2\cup\Sigma_3$ }
\label{section-result-of-F-W}
We recall results of Martin and Whitehouse in \cite{MW} (or see Feigon and Whitehouse in \cite{FW}),
where they calculated
the local integrals $J_{\pi_v'}(f_v')$ for those  $v\notin\Sigma_{2}\cup\Sigma_3$.
\subsubsection{Archimedean places}
For $v\in \Sigma_{\infty}$,  $E_v=\C$ and $D_v$ is the Hamiltonian algebra.
In this case, $\pi_{v}'=\mathrm{JL}(\pi_v)$ is a $2k_v-1$ dimensional unitary representation of $G_v'$,
which is distinguished by $(\C^\times,\Omega_v)$ with
\bna
\Omega_v(z)=\left(\frac{z}{\overline z}\right)^{m_v},\quad z\in \C^\times.
\ena
The test function is
\bna
f_v'(g)
=\left\{
\begin{aligned}
&\Omega_v(t)^{-1},\quad&& t\in E_v^\times,\\
&0,\quad&&\mbox{otherwise}.
\end{aligned}
\right..
\ena
Then one has (see lemma 3.7 in \cite{MW})
\bea
J_{\pi'_{v}}(f'_v)
=\frac{\vol(G_v')}{2k_v-1}\frac{L(1/2,\pi_{E_v}\otimes\Omega_v)L(1,\eta_v)}{L(1,\pi_v,\mathrm{Ad}) L(2,1_{F_v})}
\times\frac{1}{2\pi B(k+|m|,k-|m|)},\label{formula-spectral-Arch}
\eea
where $B(n_1,n_2)$ is the Euler beta function.

\subsubsection{Non-Archimedean places $v\in\Sigma_{1}$}
For $v\in\Sigma_1$,
$E_v/F_v$ is unramified inertia,
 $G'_v=D_v^\times/Z_v$ with $D_v$ the  quaternion
division algebra,  and $\pi_v'$ of $G_v'$ is of the form
\bna
\pi_v'=\delta_v\circ \mathrm{Nrd},
\ena
where $\delta_v$ are unramified characters of $F_v^\times$ of order at most $2$,
and $\mathrm{Nrd}$ is the reduced norm on $D_v^\times$.
Let $R_v$ be the maximal order of $D_v$ such that
\bna
E_v^\times\cap R_v=\mk o_{E_v}^\times.
\ena
On taking $f_v'=1_{Z_v R_v^\times}$, we have
\bea
 J_{\pi_v'}(f_v')
=\vol(R_v^\times Z_v/Z_v)
\frac{L(1/2,\pi_{E_v}\otimes\Omega_v)}{L(1,\pi_v,\mathrm{Ad})}.\label{formula-spectral-sigma-1}
\eea
\subsubsection{Non-Archimedean places $v\notin\Sigma_1\cup\Sigma_2\cup\Sigma_3$}
For $v\notin\Sigma_\infty\cup\Sigma_{1}\cup\Sigma_2\cup\Sigma_3$,
$\pi_v'\simeq\pi_v$ are unramified representations of $G_v'\simeq PGL_2(F_v)$.
Let $R_v$ be a maximal order of $M_2(F_v)$ such that
\bna
E_v\cap R_v=\mk o_{v}+\varpi_v^{n(\Omega_v)}\mk o_{E_v},
\ena
and take $f_v'=1_{Z_vR_v^\times}$.
We have the following result (See lemmas 2.1 and pages 172-173 in \cite{MW}).
\begin{lemma}\label{lemma-spectral-MW-1}
If $v$ is split in $E$, one has
\bna
J_{\pi_v'}(f_v')=\vol(R_v^\times Z_v/Z_v)\frac{L(1/2,\pi_{E_v}\otimes\Omega_v)}{L(1,\pi_v,\mathrm{Ad}) }
\times\left\{
\begin{aligned}
&\frac{\vol(U_v)L(2,1_{F_v})}{L(1,1_{F_v})}
,  \quad &&\mbox{$\Omega_v$ unramified}\\
&\frac{\vol(U_v)L(2,1_{F_v})L(1,1_{F_v})}{q^{n(\Omega_v)}}
, \quad  &&\mbox{$\Omega_v$ ramified}
\end{aligned}
\right..
\ena
If $v$ is non-split in $E$, one has
\bna
 J_{\pi_v'}(f_v')
=\vol(R_v^\times Z_v/Z_v)\frac{L(1/2,\pi_{E_v}\otimes\Omega_v)}{L(1,\pi_v,\mathrm{Ad})}
\times\left\{
\begin{aligned}
&\frac{1}{e(E_v|F_v)}\frac{L(2,1_{F_v})}{L(1,\eta_v)},\quad&&\mbox{$\Omega_v$ unramified}\\
&\frac{q^{-n(\Omega_v)}L(1,\eta_v)^2}{e(E_v|F_v)}\frac{L(2,1_{F_v})}
{L(1,\eta_v)},\quad&&\mbox{$\Omega_v$ ramified}
\end{aligned}
\right..
\ena
\end{lemma}

\subsection{The global integrals}
Proposition \ref{prop-spectral-side} follows from the following two lemmas.
\begin{lemma}\label{lemma-J-pi-prime-f-prim}
For $\pi'\in\mathcal F'(2{\bf k},\mk N,\Theta)$, by choosing $f'$
as in Section \ref{subsec-relative-Jacquet-chen}, we have
\bna
J_{\pi'}(f')=\bma 2\mathbf k-2\\ \mathbf k+\mathbf m-1\ema\frac{L(1/2,\pi_E\otimes\Omega)}{L(1,\pi, \mathrm{Ad})}
\frac{L_{S'(\Omega)}(1,\eta)^2}{2\sqrt{c(\Omega)d_{E/F}|\Delta_F|}}
C(2{\bf k},\mk N,\Theta,S_0),
\ena
where $S'(\Omega)=S(\Omega)-\Sigma_{2}-\Sigma_3$ and
\bea
\nonumber C(2{\bf k},\mk N,\Theta,S_0)&=&\prod_{v\in S_0}L(2,1_{F_v})\prod_{v\in S_0-\mathrm{Ram}(D)-\Sigma_3}
\vol(R_v^\times Z_v/Z_v)\\
\nonumber&&\times\prod_{v\in\Sigma_\infty}\frac{\vol(G_v')}{\pi}
\prod_{v\in\Sigma_1}\vol(R_v^\times Z_v/Z_v)(1-q_v^{-1})\\
\nonumber&&\times\prod_{v\in\Sigma_2}\vol(E_v^\times U_D^1/F_v^\times)q_v(1-q_v^{-1})(1+q_v^{-1})^2\\
&&\times\prod_{v\in\Sigma_3}\vol(E_v^\times U_\mk J^1/F_v^\times)q_v(1-q_v^{-2}).\label{lambda-2k-N-Theta-1}
\eea
Here $S_0$ is a finite set of places $F$ out of which everything is unramified.
\end{lemma}
\begin{proof}
The proof is similar as in Marin and Whitehouse \cite{MW}.
By Jacquet-Chen's factorization in Proposition \ref{prop-Jacquet-Chen}
and the results on local integrals in
formulas (\ref{formula-spectral-Arch}) and (\ref{formula-spectral-sigma-1}),
  Propositions \ref{prop-sigma_2} and \ref{prop-sigma_3}, and Lemma
 \ref{lemma-spectral-MW-1},
we have
\bna
J_{\pi'}(f')
=\frac{1}{2}\prod_{v\in S_0\atop{\mathrm{inertia}}}\epsilon(1,\eta_v,\psi_v)\frac{L(1/2,\Pi\otimes\Omega)}{L(1,\pi,\mathrm{Ad})}\prod_{v\in S_0}I_v,
\ena
where $I_v$ are given as follows.
\bit
\item
For $v\in \Sigma_\infty$,
\bna
I_v&=&\frac{\vol(G_v')}{2k_v-1}\frac{1}{2\pi B(k+|m|,k-|m|)}
L(1,\eta_v) 2L(0,\eta_v)\\
&= &L(2,1_{F_v})\frac{\vol(G_v')}{\pi}\bma2k_v-2\\k_v+m_v-1\ema.
\ena
\item
 For $v\in\Sigma_1$, $E_v/F_v$ is unramified,
 $L(s,\eta_v)=(1+q^{-s})^{-1}$ and
\bna
I_v&=&\vol(R_v^\times Z_v/Z_v)L(1,\eta_v)L(0,\eta_v)\\
&=&L(2,1_{F_v}) \vol(R_v^\times Z_v/Z_v)(1-q_v^{-1}).
\ena
\item
 For $v\in\Sigma_2$, $E_v/F_v$ is unramified,
\bna
I_v&=&\vol(E_v^\times U_D^1/F_v^\times)L(1,\eta_v)2L(0,\eta_v)\\
&=&\frac{L(2,1_{F_v})}{q_v^{n(\Omega_v)}} \vol(E_v^\times U_D^1/F_v^\times)q_v(1-q_v^{-1}).
\ena
\item
 For $v\in\Sigma_3$, $E_v/F_v$ is ramified and
\bna
I_v&=&\vol(E_v^\times U_\mk J^1/F_v^\times)L(1,\eta_v)2L(0,\eta_v)\\
&=&\frac{L(2,1_{F_v})}{q^{n(\Omega_v)}_v} \vol(E_v^\times U_\mk J^1/F_v^\times)q_v(1-q_v^{-2}).
\ena
\item
 For $v\notin \mathrm{Ram}(D)\cup\Sigma_3$ and split in $E$,
\bna
I_v=\vol(R_v^\times Z_v/Z_v)
\vol(U_v)L(2,1_{F_v})\times\left\{
\begin{aligned}
&1,\quad&&n(\Omega_v)=0\\
&q^{-n(\Omega_v)}_vL(1,\eta_v)^2,\quad&& n(\Omega_v)>0
\end{aligned}
\right..
\ena
\item
For $v\notin \mathrm{Ram}(D)\cup\Sigma_3$ and  non-split in $E$, assume $\psi_v$ is unramified,
\bna
I_v
=\vol(R_v^\times Z_v/Z_v)L(2,1_{F_v})\times
\left\{
\begin{aligned}
&1,\quad &&n(\Omega_v)=0\\
&q^{-n(\Omega_v)}_vL(1,\eta_v)^2,\quad &&n(\Omega_v)>0.
\end{aligned}
\right.
\ena
\eit
Note that
\bna
\epsilon(1,\eta_v,\psi_v)=\left\{
\begin{aligned}
&1,\quad&&v\in\Sigma_\infty,\\
&q_v^{-\frac{n(\eta_v)}{2}}\vol(U_v),\quad&&v<\infty.
\end{aligned}
\right.
\ena
This gives
\bna
\prod_{v\in S_0\atop{\mathrm{inetia}}}\epsilon(1,\eta_v,\psi_v)\prod_{v\in S_0\atop{\mathrm{split}}}\vol(U_v)
=|\Delta_F|^{-\frac{1}{2}}d_{E/F}^{-1/2}.
\ena
The result follows immediately.
\end{proof}

\begin{lemma}
The constant $C(2{\bf k},\mk N,\Theta,S_0)$ in (\ref{lambda-2k-N-Theta-1}) is
\bea
C(2{\bf k},\mk N,\Theta,S_0)
=\frac{4^{[F:\Q]}}{|\Delta_F|^{\frac{3}{2}}}
\frac{2^{\#\Sigma_3}}{|\mk N_1||\mk N_3|}.
\label{lambda-2k-N-Theta-2}
\eea
\end{lemma}
\begin{proof}
By Proposition \ref{prop-local-Tamagawa}, we have the following.
\bit
\item
For $v\in\Sigma_\infty$, we have
\bna
\frac{\vol(G_v')}{\pi}=4\pi=4L(2,1_{F_v})^{-1}.
\ena
\item
For $v\in\Sigma_1$,
\bna
\vol(R_v^\times Z_v/Z_v)=L(2,1_{F_v})^{-1}\vol(U_{v})^3
\frac{1}{q_v-1}.
\ena
\item
For $v\in\Sigma_2$,
$E_v^\times U_D^1=Z_v R_v^\times$ and we have
\bna
\vol(E_v^\times U_D^1/F_v^\times)=
L(2,1_{F_v})^{-1}\vol(U_{v})^3
\frac{1}{q_v-1}.
\ena
\item
For $v\in\Sigma_3$,
\bna
\vol(U_\mk J^1)&=&\frac{1}{(q_v-1)^2}\vol (U_\mk J)\\
&=&\frac{1}{(q_v-1)^2}\frac{1}{q_v+1}L(2,1_{F_v})^{-1}\vol(U_{v})^4.
\ena
Note that
\bna
E_v^\times U_\mk J^1/F_v^\times=\frac{\langle\varpi_{E_v}\rangle k_{E_v}^\times U_{E_v}^1 U_\mk J^1}
{\langle\varpi_v\rangle k_{F_v}^\times U_v^1}
=\frac{\langle\varpi_{E_v}\rangle}{\langle\varpi_v\rangle}\frac{ U_\mk J^1}{U_v^1}.
\ena
It gives
\bna
\vol(E_v^\times U_\mk J^1/F_v^\times)&=&2 (q_v-1)\frac{1}{(q_v-1)^2}\frac{1}{q_v+1}
L(2,1_{F_v})^{-1}\vol(U_{v})^3\\
&=&\frac{2}{q_v^2-1}L(2,1_{F_v})^{-1}\vol(U_{v})^3.
\ena
\item
For $v\in S_0$ and $v\notin \mathrm{Ram}(D)\cup\Sigma_3$, we know $G'(F_v)\simeq PGL_2(F_v)$
and $R_v^\times\simeq GL_2(\mk o_v)$. Thus
\bna
\vol(R_v^\times Z_v/Z_v)=L(2,1_{F_v})^{-1}\vol(U_{v})^3.
\ena
\eit
Therefore (\ref{lambda-2k-N-Theta-2}) follows immediately from the above calculation.
\end{proof}

\section{Local Orbital integrals}\label{sec-local-orbital-integrals}
For  orbital integrals $I(0,f')$ in (\ref{I-0-f-prime}) and $I(\xi, f')$
in (\ref{I-xi-f-prime}), we have
\bna
I(0,f')=\prod_v I(0,f_v'),\quad
I(\xi,f')=\prod_v I(\xi,f'_v),
\ena
where
\bea
I(0,f_v')=\int_{F_v^\times\backslash E_v^\times}f_v'(t)\Omega_v(t)dt\label{local-singular-orbital-integral}
\eea
are local singular orbital integrals, and
\bea
I(\xi,f'_v)
=
\int_{F_v^\times\backslash E_v^\times}\Omega(\alpha)
\left\{\int_{E_v^1}
f_v'\left(\bma \alpha&\\&\overline{\alpha}\ema
\bma1&\epsilon_v z x\\\overline{z x}&1\ema \right)
 dz\right\}d\alpha\label{local-regular-orbital-integral}
\eea
are local regular orbital integrals.
\subsection{Local singular orbital integrals}
For $I(0,f_v')$ in (\ref{local-singular-orbital-integral}), we have the following result.
\begin{lemma}\label{lemma-local-singular-oribtal}
For $v\in Ram(D)\cup\Sigma_3$,
\bna
I(0,f_v')=\vol(F_v^\times\backslash E_v^\times).
\ena
For $v\notin \mathrm{Ram}(D)\cup\Sigma_3$,
\bna
I(0,f_v')=\vol(F_v^\times\backslash F_v^\times (\mk o_{v}+\varpi_v^{n(\Omega_v)}\mk o_{E_v})^\times).
\ena
\end{lemma}
\begin{proof}
For $v\notin\Sigma_{2}\cup\Sigma_3$, the results are
Lemmas 4.1, 4.6 and 4.7 in \cite{FW}. For $v\in\Sigma_{2}\cup\Sigma_{3}$, the results follow immediately
from the choice of $f_v'$ and $\Omega_v$.
\end{proof}

\subsection{Local regular orbital integrals}
For $I(\xi,f_v')$ in (\ref{local-regular-orbital-integral}),
we have the following result.
\begin{lemma}\label{lemma-1}For $v\in\Sigma_\infty$ and $\xi\in F_v$ with $\xi<0$ we have
\bna
I(\xi,f_v')=
\frac{\vol(F_v^\times\backslash E_v^\times)^2}{(1-\xi)^{k_v-1}}
\sum
_{i=0}^{k_v-|m_v|-1}\bma k_v-m_v-1\\i\ema\bma k_v+m_v-1\\i\ema(-\xi)^{i}.
\ena
For $v\in\Sigma_1\cup\Sigma_2\cup\Sigma_3$, we have
\bna
I(\xi,f_v')=\left\{
\begin{aligned}
&\vol(F_v^\times\backslash E_v^\times)^2,\quad &&v(\xi)\geq 1\\
&0,\quad&&v(\xi)\leq 0.
\end{aligned}
\right.
\ena
\end{lemma}

\begin{proof}
For $v\in\Sigma_{\infty}\cup\Sigma_{1}$, the results are Lemmas 4.10 and 4.14 in \cite{FW}.

Assume $v\in\Sigma_2$.
Without loss of generality,
we can assume that $v(\epsilon)=1$ so that $\varpi_v=\epsilon$ is a prime element in $F_v^\times$.
By (\ref{formula-maximal-order}), the maximal order $R_v$  with $R_v\cap E_v=\mk o_{E_v}$
is of the form
\bna
R_v=
\left\{\bma\alpha&\varpi_v\beta\\\overline\beta&\overline\alpha
\ema,
\quad\alpha,\beta\in\mk o_{E_v}\right\}.
\ena
Let $\varpi_D=\bma&\varpi_v\\1&\ema$. For $n\geq 1$ we  let
\bna
U_D^{n}=1+\varpi_D^nR_v.
\ena
Recall $f_v'$ in (\ref{testfunc-v-in-Sigma-2}).
Given $\xi=\epsilon x\overline x$ and $z\in E_v^1$, the element
\bna
\bma1&\varpi_v z x\\\overline{zx}&1\ema = I_2+\bma&\varpi_v\\1\ema\bma zx\\&\overline {zx}\ema
\ena
is in $U_D^1$,
if and only if
\bna
v_D\bma zx\\&\overline {zx}\ema=v(N_{E_v/F_v}(x))\geq 0,
\ena
or equivalently, $v(\xi)=v(\varpi_v N_{E_v/F_v}(x))\geq 1$, in which case
\bna
I(\xi,f_v')=\vol(F_v^\times\backslash E_v^\times)^2.
\ena

Assume $v\in\Sigma_3$. We know $E_v/F_v$ is ramified,
$\mk o_{E_v}=\mk o_{v}[\varpi_{E_v}]$ and $D_v\simeq M_2(F_v)$.
Without loss of generality, we assume $\epsilon=1$.
Given $\xi=\epsilon x\overline x$ and $z\in E_v^1$, we can write
\bna
zx=a_{zx}+\varpi_{E_v}b_{zx},\quad a_{zx},b_{zx}\in \mk o_v.
\ena
Recall $f'_v$ in (\ref{testfunc-v-in-Sigma-3}).
By  the isomorphism
$D_v\simeq M_2(F_v)$ defined in (\ref{iso-Dv-M2Fv}), the image of
\bna
\bma 1&\epsilon xz\\\overline{xz}&1\ema
=I_2+ \bma&a_{zx}+\varpi_{E_v} b_{zx}\\a_{zx}-\varpi_{E_v}b_{zx}\ema
\ena
in $M_2(F_v)$ is
\bna
\gamma(x,z):=I_2+\bma a_{zx}&- b_{zx}\\b_{zx}\varpi_v&-a_{zx}\ema.
\ena
Note that $\gamma(x,z)\in U_{\mk J}^1$
if and only if $a_{zx}\in\mk p_v$ and $b_{zx}\in \mk o_v$, i.e.
\bna
zx\in\varpi_{E_v}\mk o_{E_v},
\ena
which is equivalent to $v(\xi)=v(\epsilon N_{E_v/F_v}(zx))\geq 1$.
In this case,
\bna
\gamma(x,z)\equiv I_2+\bma&-b_{zx}\\b_{zx}\varpi_v\ema\bmod
\left[\begin{matrix}\mk p_v&\mk p_v\\ \mk p_v^2&\mk p_v\end{matrix}\right]
\ena
and thus
\bna
f_v'(\gamma(x,z))
=\theta_{v,\pm,2\beta}\left(2\beta\frac{(-b_{zx}+b_{zx})\bmod \mk p_v}{2}\right)=1.
\ena
This proves the case $v\in\Sigma_3$.
\end{proof}

The above Lemma gives the explicit calculation of $I(\xi, f_v')$ at $v\in \mathrm{Ram}(D)\cup\Sigma_3$.
For $v\notin \mathrm{Ram}(D)\cup\Sigma_3$, we have the following results (see \cite{FW}).
\begin{lemma}\label{lemma-3}
Assume $v\notin \mathrm{Ram}(D)\cup\Sigma_3$.
If $v(1-\xi)>v(\mk d_{E/F} \mk c(\Omega))$, then $I(\xi,f_v')=0$.
\end{lemma}
As in \cite{FW}, this lemma is used in obtaining the condition in Theorem \ref{thm-non-vanihsing}
under which all global regular orbital integrals vanish. To obtain the subconvexity bound,
we need more information on local regular orbital integrals at $v\not\in \mathrm{Ram}(D)\cup \Sigma_3$.

\begin{lemma}\label{lemma-2-1}
Let $v\notin \mathrm{Ram}(D)\cup\Sigma_3$ be not split in $E$. Assume $n(\Omega_v)=0$.
If $v$ is ramified in $E$, we assume that the characteristic of the residue field $k_{F_v}$ is odd.
Then
\bna
I(\xi,f_v')=\vol(F_v^\times\backslash F_v^\times U_{E_v})\vol(F_v^\times\backslash E_v^\times)
\Omega_v(\varpi_{E_v}^{v_E(1-\xi)/2})
\times\left\{
\begin{aligned}
&0,\quad &&v(1-\xi)>v(\mk d_{E/F})\\
&1,\quad &&v(1-\xi)\leq 0\\
&\frac{1}{2},\quad &&v(1-\xi)=v(\mk d_{E/F})>0.
\end{aligned}
\right.
\ena
\end{lemma}
\begin{lemma}\label{lemma-2-2}
Let $v$ be a finite valuation which is split in $E$ and such that $n(\Omega_v)=0$. Then
\bna
I(\xi,f_v')=\vol(U_{v})^2\times
\left\{
\begin{aligned}
&0,\quad&&v(1-\xi)>0\\
&1+v(\xi),\quad &&v(1-\xi)=0\\
&\Omega_v(\xi,1)\sum_{l=0}^{|v(\xi)|}\Omega_v(\varpi_v^{2l},1),\quad &&v(1-\xi)<0.
\end{aligned}
\right.
\ena
\end{lemma}

\begin{lemma}\label{lemma-bound-othercase}
Assume that $n(\Omega_v)>0$. Let $k=v(1-\xi)/2$.
Then there exists a constant $C(E_v,F_v)$ which is equal to $1$ for all $v$ unramified in $E$ and such that
\bna
\begin{aligned}
|I(\xi, f_v')|
\leq& q_v^{-n(\Omega_v)}L(1,\eta_v) \vol(U_{v}\backslash U_{E_v})
\vol(E_v^1\cap U_{E_v}) C(E_v,F_v)\\
&\times\left\{
\begin{aligned}
&q_v^{-k}L(1,\eta_v),\quad &&k>0\\
&1,\quad&&\mbox{$k\leq 0$ and $v$ is not split},\\
&1+|v(\xi)|,\quad&&\mbox{$k\leq 0$ and $v$ is split}.
\end{aligned}
\right.
\end{aligned}
\ena
\end{lemma}

\section{Global orbital integrals}\label{sec-global-orbital-integrals}
In this section, we calculate the global orbital integrals and prove Theorems \ref{thm-non-vanihsing}
and \ref{thm-subconviexity-0}.
\subsection{Global orbital integrals}
Consider $I(0,f')$ in (\ref{I-0-f-prime}). By Lemma \ref{lemma-local-singular-oribtal},
\bna
I(0,f')
=\prod_{v\in \mathrm{Ram}(D)\cup\Sigma_3}\vol(F_v^\times\backslash E_v^\times)
\prod_{v<\infty\atop{v\notin\Sigma}}\vol(F_v^\times\backslash F_v^\times(1+\varpi_v^{n(\Omega_v)}\mk o_{E_v})^\times).
\ena
Recall that
\bna
\vol(F_v^\times\backslash E_v^\times)=\vol(\R^\times\backslash \C^\times)=2,\quad \mbox{for $v\in\Sigma_\infty$},
\ena
and
\bna
\vol(F_v^\times\backslash F_v^\times(1+\varpi_v^{n(\Omega_v)}\mk o_{E_v})^\times)
&=&\left\{
\begin{aligned}
&\vol(U_{v}\backslash U_{E_v}),\quad &&v\notin S'(\Omega),\\
&\vol(U_{v}\backslash U_{E_v})q_v^{-n(\Omega_v)}L(1,\eta_v),\quad &&v\in S'(\Omega).
\end{aligned}
\right.
\ena
Thus
\bea
I(0,f')=2^{[F:\Q]}\sqrt{\frac{|\Delta_F|}{|\Delta_E|}}\frac{|\mk N_2\mk N_3|}{\sqrt{c(\Omega)}}L_{S'(\Omega)}(1,\eta)
\label{formula-global-singular-orbital-integral}.
\eea

Consider $I(\xi,f')$.
By lemmas \ref{lemma-1} and \ref{lemma-3}, we know
 $I(\xi,f')$ vanishes outside of the set
\bna
\mathcal S(\Omega,\mk N)
=
\left\{\xi\in \epsilon N E^\times,\quad
\left.
\begin{aligned}
&\mbox{$v(\xi)\geq 1$ for $v\in\Sigma_{1}\cup\Sigma_2\cup\Sigma_3$},\\
&(1-\xi)^{-1}\in
\left(\mk c(\Omega)
\mk d_{E/F}\right)^{-1}\mk N_2^2\mk N_3^3
\end{aligned}
\right.
\right\}.
\ena
\begin{lemma}\label{lemma-finiteness}The set $\mathcal S(\Omega,\mk N)$ is finite.
Moreover,
it is empty if either
\bna
|\mk N_1\mk N_2^{1+2h_F}\mk N_3^{1+3h_F}|\geq d_{E/F}^{h_F}c(\Omega)^{h_F}
\ena
or
\bna
|\mk N_1\mk N_2^3\mk N_3^4|\geq d_{E/F}c(\Omega)\sqrt{|\Delta_F|}.
\ena
\end{lemma}
\begin{proof}The proof is similar as the proof of lemma 4.21 in \cite{FW}.
Let $\iota_v:F\rightarrow\R$ be the real embedding for $v\in\Sigma_\infty$.
Then $1\neq\xi\in\epsilon NE^\times$ if and only if  $1\neq\xi\in F^\times$ satisfies
\bit
\item [1.] $\iota_v(\xi)<0$ for all $v\in\Sigma_\infty$,
\item [2.] $v(\xi)$ is odd for $v\in\Sigma_{1}\cup\Sigma_{2}$,
\item [3.] for $v\notin \mathrm{Ram}(D)$, $\eta_v(\xi)=1$.
\eit
For $\xi\in \mathcal S(\Omega,\mk N)$, we know $(1-\xi)^{-1}x\in\mk o_F$
for any $0\neq x\in \mk c(\Omega)
\mk d_{E/F}\mk N_2^{-2}\mk N_3^{-3}$.
There exists $m\in\mk o_F$ such that
\bna
\xi=1-\frac{x}{m}=\frac{m-x}{m}.
\ena
Let $y=m-x$. The condition $\iota_v(\xi)<0$ for any $v\in\Sigma_\infty$ implies that
\bna
|\iota_v(y)|< |\iota_v(x)|,\quad \forall v\in\Sigma_\infty.
\ena
Thus the finiteness of $\mathcal S(\Omega,\mk N)$ follows from the finiteness of
the set
\bna
\left\{y\in \mk o_F:|\iota_v(y)|<|\iota_v(x)|,
\quad\forall v\in\Sigma_\infty\right\}.
\ena

Note that $\xi\in\epsilon N(E^\times)$ also satisfies
\bna
v(\xi)\geq 1,\quad\forall v\in\Sigma_{1}\cup\Sigma_2\cup\Sigma_3.
\ena
Thus $\mathcal S(\Omega,\mk N)$ is empty if
\bea
\left\{y\in \mk o_F:|\iota_v(y)|<|\iota_v(x)|,
\begin{aligned}
&\quad\forall v\in\Sigma_\infty,\\
&\forall 0\neq x\in
(\mk N_2^2\mk N_3^3)^{-1}\left(\mk c(\Omega)
\mk d_{E/F}\right)
\end{aligned}
\right\}\bigcap\mk N_1\mk N_2\mk N_3=\{0\}.\label{key-cond}
\eea
To obtain sufficient conditions such that (\ref{key-cond}) holds, we replace the integral ideal
\bna
\mk b:=(\mk N_2^2\mk N_3^3)^{-1}\left(\mk c(\Omega)
\mk d_{E/F}\right)
\ena by a principle integral ideal $\mk a=(a)\subset\mk b$.
On taking
$\mk a=\mk b^{h_F}$,
we obtain the first sufficient condition.
By Lemma 6.2 in page 35 in \cite{Ne},
there exists $0\neq a\in\mk b$
such that
\bna
|N_{F/\Q}(a)|\leq \sqrt{|\Delta_F|}|\mk b|.
\ena
Thus (\ref{key-cond}) is empty if
\bna
|\mk N_1\mk N_2\mk N_3|\geq \frac{c(\Omega)d_{E/F}}{|\mk N_2|^2|\mk N_3|^3}\sqrt{|\Delta_F|},
\ena
which gives the second sufficient condition.
\end{proof}

\subsection{Proof of Theorems \ref{thm-non-vanihsing}}
By Proposition \ref{prop-spectral-side}, the spectral side of the relative trace formula gives
\bna
J(f')=
\frac{4^{[F:\Q]}L_{S'(\Omega)}(1,\eta)^2}{2|\Delta_F|^{2}\sqrt{c(\Omega)d_{E/F}}}
\frac{2^{\#\Sigma_3}}{|\mk N_1\mk N_3|}
\bma 2{\bf k}-2\\ {\bf k+m}-1\ema
\sum_{\pi\in\mathcal F^{\mathrm{new}}(2{\bf k},\mk N,\Theta)}\frac{L(1/2,\pi_E\otimes\Omega)}{L(1,\pi,\mathrm{Ad})}.
\ena
On the geometric side,
by (\ref{formula-geometric-side-pre}) and (\ref{formula-global-singular-orbital-integral}),
we have
\bna
J(f')=L(1,\eta)
2^{[F:\Q]+1}\sqrt{\frac{|\Delta_F|}{|\Delta_E|}}\frac{|\mk N_2\mk N_3|}{\sqrt{c(\Omega)}}L_{S'(\Omega)}(1,\eta)
+\sum_{\xi\in\epsilon N(E^\times)}I(\xi,f'),
\ena
where we have used the fact that
$\vol(\A_F^\times E^\times\backslash \A_E^\times)=2L(1,\eta)$.
By Lemma \ref{lemma-finiteness}, all regular orbital integrals $I(\xi,f')$
 vanish if one of the conditions in Theorem \ref{thm-non-vanihsing} is satisfied.
 This proves Theorem \ref{thm-non-vanihsing}.
\subsection{Proof of Theorem 2}
We recall the conductor $c(\pi\times\sigma_\Omega)$ for $\pi$ and $\Omega$ as in Theorem \ref{thm-subconviexity-0}.
By local Langlands correspondence in Sections \ref{sec-sigma-2} and \ref{sec-sigma-3},
the local conductors of $\pi_v\times\sigma_{\Omega_v}$ at the joint ramified places are
\bna
c(\pi_v\times\sigma_{\Omega_v})=
\left\{
\begin{aligned}
&q_v^2,&\quad& v\in\Sigma_2,\\
&q_v^5,&\quad& v\in\Sigma_3.
\end{aligned}
\right.
\ena
The conductors of $\pi$ and $\sigma_\Omega$ are
\bna
c(\pi)=|\mk N_1\mk N_2^2\mk N_3^3|,\quad c(\sigma_\Omega)=d_{E/F}c(\Omega),
\ena
respectively.
Thus we have
\bna
c(\pi\times\sigma_{\Omega})
=|\mk N_1\mk N_2|^2|\mk N_3|^5 \prod_{v\notin\Sigma_{1}\cup\Sigma_2\cup\Sigma_3}\left(d_{E_v/F_v}c(\Omega_v)\right)^2.
\ena
Assume $\Sigma_3=\emptyset$.
The convexity bound of $L_{\mathrm{fin}}(1/2,\pi\times\sigma_\Omega)$ is
\bea
L_{\mathrm{fin}}(1/2,\pi\times\sigma_{\Omega})\ll_{\mathbf{k},\epsilon}(|\mk N_1\mk N_2|d_{E/F}c'(\Omega))^{\frac{1}{2}+\epsilon},
\label{formula-convexity-bound}
\eea
where
$c'(\Omega)$ is defined in (\ref{c-prime-Omega}).

To prove Theorem \ref{thm-subconviexity-0}, by the result on spectral side in Proposition \ref{prop-spectral-side},
we have
\bna
J(f')\gg_{E,F,{\bf k},\epsilon}\frac{L_{S'(\Omega)}(1,\eta)^2}{\sqrt{c'(\Omega)}|\mk N_1\mk N_2|}
\sum_{\pi\in\mathcal F(\mk N,{2\bf k},\Theta)}\frac{L_{\mathrm{fin}}(1/2,\pi\times\sigma_\Omega)}{L(1,\pi,\mathrm{Ad})}.
\ena
On geometric side,
by (\ref{formula-global-singular-orbital-integral}), the irregular orbital integral
contributes to
\bna
\vol(\A_F^\times E^\times\backslash \A_E^\times)I(0,f')
\ll_{E,F,\epsilon} c'(\Omega)^{-1/2+\epsilon}.
\ena
We claim that all regular orbital integrals contribute to
\bea
I_{\mathrm{reg}}(f')=\sum_{\xi\in\epsilon N_{E/F}(E^\times)}I(\xi,f')\ll_{E,F,\mathbf{k},\epsilon}\frac{c'(\Omega)^\epsilon}{|\mk N_1\mk N_2|}.
\label{formula-bound-regular-orbital-integrals}
\eea
Therefore,
\bna
\frac{L_{S'(\Omega)}(1,\eta)^2}{\sqrt{c'(\Omega)}|\mk N_1\mk N_2|}
\sum_{\pi\in\mathcal F(\mk N,{2\bf k},\Theta)}\frac{L_{\mathrm{fin}}(1/2,\pi\times\sigma_\Omega)}{L(1,\pi,\mathrm {Ad})}
\ll_{E,F,{\bf k},\epsilon} c'(\Omega)^{-\frac{1}{2}+\epsilon}+\frac{c'(\Omega)^{\epsilon}}{|\mk N_1\mk N_2|}.
\ena
By the non-negativity of  $L(1/2,\pi\times\sigma_\Omega)$, the positivity of $L(1,\pi,\mathrm{Ad})$, and the bound
\bna
L(1,\pi,\mathrm{Ad})\gg_{\mathbf{k},\epsilon}|\mk N_1\mk N_2|^{-\epsilon},\qquad
L_{S'(\Omega)}(1,\eta)\gg_{\epsilon}c'(\Omega)^{-\epsilon},
\ena
we finally obtain
\bna
L_{\mathrm{fin}}(1/2,\pi\times\sigma_\Omega)\ll_{E,F,{\bf k},\epsilon} |\mk N_1\mk N_2|^{1+\epsilon} c'(\Omega)^\epsilon
+|\mk N_1\mk N_2|^{\epsilon} c'(\Omega)^{1/2+\epsilon}.
\ena
This proves Theorem \ref{thm-subconviexity-0}.
\subsection{Proof of (\ref{formula-bound-regular-orbital-integrals})}
We still need to prove (\ref{formula-bound-regular-orbital-integrals}).
The proof is similar as that of Lemma 6.7 in Feigon and Whitehouse
\cite{FW}.
To give a self-contained proof, we proceed as follows.

By Lemma 6.2 in page 35 in \cite{Ne}, we can find
$0\neq a  \in \mk d_{E/F}\mk c(\Omega)\mk N_2^{-2}$ with
\bea
|N_{F/\Q}(a)|\leq \frac{d_{E/F}c(\Omega)}{|\mk N_2|^2}\sqrt{|\Delta_F|},\label{formula-condition-satisfied-by-a}
\eea
and then enlarge $\mathcal S(\mk N,\Omega)$ to be
\bna
\mathcal S(a)
=\left\{\xi_y=\frac{y}{y+a}\in\epsilon N(E^\times),\quad y\in\mk N_1\mk N_2
\right\}.
\ena
Thus we have
\bna
I_{\mathrm{reg}}(f')=\sum_{\xi_y\in \mathcal S(a)}I(\xi_y,f'^{S'(\Omega)})\prod_{v\in S'(\Omega)}I(\xi_y,f'_v).
\ena
\begin{lemma}\label{lemma-I-xi-y-up-S-prime-Omega} For $\xi_y\in \mathcal S(a)$, we have
\bna
I(\xi_y,f'^{S'(\Omega)})\ll_{E,\mathbf{k},\epsilon} |N_{F/\Q}(a)|^\epsilon
\ena
for any $\epsilon>0$.
\end{lemma}
\begin{proof} The proof is similar as that of Lemma 6.7 in \cite{FW}.
For $\mk a\subset\mk o_F$ an integral ideal, we define
\bna
R_{E_v}(\mk a)=\left\{
\mk b_v\subset \mk o_{E_v},
N_{E_v/F_v}(\mk b_v)=\mk a\mk o_{v}\right\}
\ena
and denote $|R_{E_v}(\mk a)|=\#R_{E_v}(\mk a)$.
By Lemmas \ref{lemma-1}, \ref{lemma-2-1} and \ref{lemma-2-2},
we have the following
(see page 389-390 in \cite{FW}).
\bit
\item If $v\in\Sigma_\infty$,
\bna
I(\xi_y,f_v')\ll_{\mathbf{k}}1.
\ena
\item If $v\in\Sigma_{1}\cup\Sigma_2$,
\bna
I(\xi_y,f_v')=
\vol(U_v\backslash U_{E_v})^2
 |R_{E_v}(y(\mk N_1\mk N_2)^{-1})|
\ena
\item If $v\notin\Sigma_\infty\cup\Sigma_{1}\cup\Sigma_2$ and is inertia in $E$ and $n(\Omega_v)=0$,
\bna
I(\xi_y,f_v')=\vol(U_v\backslash U_{E_v})^2|R_{E_v}((y))||R_{E_v}((y+a))|.
\ena
\item If $v\notin\Sigma_\infty\cup\Sigma_{1}\cup\Sigma_2$  and is  split in $E$ and $n(\Omega_v)=0$,
\bna
|I(\xi_y,f_v')|\leq \vol(U_v\backslash U_{E_v})^2|R_{E_v}((y))||R_{E_v}((y+a))|.
\ena
\item If $v\notin\Sigma_\infty\cup\Sigma_{1}\cup\Sigma_2$ and is ramified in $E$ and $n(\Omega_v)=0$,
\bna
|I(\xi_y,f_v')|\leq 2 \vol(U_v\backslash U_{E_v})^2|R_{E_v}((y))||R_{E_v}((y+a))|.
\ena
\eit
Thus one has
\bna
I(\xi_y,f'^{S'(\Omega)})\ll_{E,\mathbf{k}} |R^{S'(\Omega)}(y\mk N_1^{-1}\mk N_2^{-1})| |R^{S'(\Omega)}((y+a))|,
\ena
where $R^{S'(\Omega)}(\mk a)$ is defined by
\bna
R^{S'(\Omega)}(\mk a)=\left\{\mk b\in\mk o_{E},\quad
\begin{aligned}
& N_{E_v/F_v}(\mk b\mk o_{E_v})=\mk a \mk o_v \,\,\mbox{for $v\notin S'(\Omega)$}\\
& \mk p_v\nmid \mk b \,\,\mbox{for $v\in S'(\Omega)$}
\end{aligned}\right\}.
\ena
Note that
$|R^{S'(\Omega)}(\mk a)|\ll_\epsilon  |N_{F/\Q}(\mk a)|^\epsilon$.
The condition $\xi_y\in\epsilon N(E^\times)$ implies that
$|\iota_v(y)|<|\iota_v(a)|$ and $|\iota_v(y+a)|<|\iota_v(a)|$
for all $v\in\Sigma_\infty$. This gives
\bna
I(\xi_y,f'^{S'(\Omega)})\ll_{E,\mathbf{k},\epsilon} |N_{F/\Q}(a)|^\epsilon.
\ena
\end{proof}
By
Lemma \ref{lemma-I-xi-y-up-S-prime-Omega}, we have
\bna
I_{\mathrm{reg}}(f')\ll_{E,F,\mathbf{k},\epsilon}|N_{F/\Q}(a)|^\epsilon \sum_{\xi_y\in \mathcal S(a)}\prod_{v\in S'(\Omega)}|I(\xi,f_v')|.
\ena
Assume $S'(\Omega)=\{v_{\mk p_i},1\leq i\leq m\}$.
We can partition $\mathcal S(a)$
to be union of subsets
\bna
\mathcal S_{(r_i),(t_i)}(a)=
\left\{\xi_y=\frac{y}{y+a}\in\epsilon N(E^\times):\quad
\begin{aligned}
&y\in\mk N_1\mk N_2, \\
&v_{\mk p_i}(y)=r_i, v_{\mk p_i}(y+a)=t_i.
\end{aligned}
\right\},\quad 1\leq i\leq m,
\ena
and thus
\bna
I_{\mathrm{reg}}(f')\ll_{E,F,\mathbf{k},\epsilon} |N_{F/\Q}(a)|^\epsilon
\sum_{r_1\geq 1}\cdots\sum_{r_m\geq 1}\sum_{t_1\geq 1}\cdots\sum_{t_m\geq m}
\sum_{\xi_y\in\mathcal S_{(r_i),(t_i)}(a)}
\prod_{i=1}^m|I(\xi_y,f_{\mk p_i}')|.
\ena
Note that only finite number of $\mathcal S_{(r_i),(t_i)}(a)$ are non-empty. More precisely,
we have the following lemma (see Lemma 6.6 in \cite{FW}).

\begin{lemma}\label{lemma-s-ri-ti-a}
$\mathcal S_{(r_i),(t_i)}(a)$ is empty unless for each $i=1,\cdots m$,
\bit
\item[(1)] $r_i<v_{\mk p_i}(a)$ and $t_i=r_i$, or
\item[(2)] $r_i>v_{\mk p_i}(a)$ and $t_i=v_{\mk p_i}(a)$, or
\item[(3)] $r_i=v_{\mk p_i}(a)$ and $t_i\geq v_{\mk p_i}(a)$.
\eit
In addition,
\bna
|\mathcal S_{(r_i),(t_i)}(a)|\leq \frac{2^{[F:\Q]+1}|N_{F/\Q}(a)|}{\left|N_{F/\Q}(\mk N_1\mk N_2 \mk p_1^{\max\{r_1,t_1\}}\cdots \mk p_m^{\max\{r_m,t_m\}})\right|}.
\ena
\end{lemma}
We need to bound $I(\xi_y,f_{v_{i}}')$ for $\xi_y\in \mathcal S_{(r_i),(t_i)}(a)$ and $v_i=v_{\mk p_i}\in S'(\Omega)$.
Note that
\bna
v_i(\xi_y)=r_i-t_i,\quad v_i(1-\xi_y)=v_i(a)-r_i.
\ena
By Lemma \ref{lemma-bound-othercase}, we have
\bna
|I(\xi_y,f'_{v_{i}})|
\ll_{E,F} q_{v_i}^{-n(\Omega_{v_i})}
\left\{
\begin{aligned}
&L(1,\eta_{v_i})^2 q_{v_i}^{-\frac{v_{i}(a)-r_i}{2}},\quad &&0\leq r_i<v_{i}(a)\\
&L(1,\eta_{v_i})
(1+t_i-r_i),\quad&& r_i=v_{i}(a)\\
&L(1,\eta_{v_i})
(1+r_i-v_{i}(a)),\quad&& r_i>v_{i}(a)
\end{aligned}
\right..
\ena
This together with Lemma \ref{lemma-s-ri-ti-a} gives
\bna
\sum_{r_1,t_1\geq 0}\cdots\sum_{r_m,t_m\geq 0}
\sum_{\xi_y\in \mathcal S_{(r_i),(t_i)}(a)}\prod_{i=1}^m
|I(\xi_{y},f_{v_i}')|
\ll_{E,F}\frac{|N_{F/\Q}(a)|}{|\mk N_1\mk N_2|}
\prod_{i=1}^m
q_{v_{i}}^{-n_i(\Omega_{v_i})-\frac{v_{i}(a)}{2}}
\ena
and thus
\bna
I_{\mathrm{reg}}(f')\ll_{E,F,\mathbf{k},\epsilon}
\frac{|N_{F/\Q}(a)|^{1+\epsilon}}{|\mk N_1\mk N_2|}
\prod_{i=1}^m
q_{v_i}^{-n(\Omega_{v_i})-\frac{v_{i}(a)}{2}}.
\ena
Note that $0\neq a  \in \mk d_{E/F}\mk c(\Omega)\mk N_2^{-2}$ and $a$ satisfies (\ref{formula-condition-satisfied-by-a}).
We have $v_{i}(a)\geq  v_{i}(\mk c(\Omega_{v_i}))$ for $1\leq i\leq m$ and thus
\bna
I_{\mathrm{reg}}(f')
&\ll_{E,F,{\bf k},\epsilon}&
\frac{|N_{F/\Q}(a)|^{1+\epsilon}}{\left|\mk N_1\mk N_2\right|}
\prod_{i=1}^m
q_i^{-2n(\Omega_i)}\\
&\ll_{E,F,{\bf k},\epsilon}&\frac{c'(\Omega)^\epsilon}{|\mk N_1\mk N_2|}.
\ena
This proves (\ref{formula-bound-regular-orbital-integrals}).

\bigskip
\noindent{\bf Acknowledgments}
\bigskip

This work was completed when the author visited University of Minnestoa
in  2016-2017.
The author would like to express his thanks to Professor Dihua Jiang
for suggesting this question and
constant encouragement.
The author is also grateful to Bin Xu, Lei Zhang and Yongqiang Zhao for valuable discussions.
The author is supported in part by the Natural Science Foundation
of Shandong Province (Grant No. ZR2014AQ002),
Innovative Research Team in University (Grant No. IRT16R43)
and  China Scholarship Council.

\end{document}